\documentclass[a4paper,10pt]{amsart}

\usepackage[english]{babel}

\usepackage{lmodern}

\usepackage[utf8]{inputenc}

\usepackage{amsfonts}
\usepackage{amsthm} 
\usepackage{amsmath}
\usepackage{amssymb}

\usepackage{tikz}

\usetikzlibrary{arrows}
\usetikzlibrary{matrix}

\theoremstyle{definition}
\newtheorem{defi}{Definition}[section] 

\theoremstyle{plain}
 
\newtheorem{prop}[defi]{Proposition}
\newtheorem{theo}[defi]{Theorem}
\newtheorem{corollary}[defi]{Corollary}

\theoremstyle{remark}
\newtheorem{remark}[defi]{Remark}
\newtheorem{exemple}[defi]{Example}
\newtheorem{question}[defi]{Question}

\newcommand{\id}{\mathrm{id}}
\newcommand{\Id}{\mathrm{Id}}

\newcommand{\C}{\mathbb{C}}
\newcommand{\R}{\mathbb{R}}

\newcommand{\QQ}{\mathcal{Q}}
\newcommand{\Z}{\mathbb{Z}}
\newcommand{\N}{\mathbb{N}}
\newcommand{\F}{\mathbb{F}}
\newcommand{\NN}{\mathcal{N}}

\title[Link invariants from $L^2$-Burau maps of braids]{Link invariants from $L^2$-Burau maps of braids}
\author{Fathi Ben Aribi}

\address{
	UCLouvain, IRMP, Chemin du Cyclotron 2 \\
	1348 Louvain-la-Neuve \\
	Belgium}
\email{fathi.benaribi@uclouvain.be}

\makeatletter
\@namedef{subjclassname@2020}{%
	\textup{2020} Mathematics Subject Classification}
\makeatother

\subjclass[2020]{57K10; 57M05; 20F36; 47C15}
\keywords{$L^2$-torsion; braid groups; Fuglede--Kadison determinant; invariants of links; Markov moves; Fox calculus}

\begin{document}

\maketitle

\begin{abstract}

A previous work of A. Conway and the author introduced $L^2$-Burau maps of braids, which are generalizations of the Burau representation whose coefficients live in a more general group ring than the one of Laurent polynomials. This same work established that the $L^2$-Burau map of a braid at the group of the braid closure yields the $L^2$-Alexander torsion of the braid closure in question, as a variant of the well-known Burau--Alexander formula.

In the present paper, we generalize the previous result to $L^2$-Burau maps defined over all quotients of the group of the braid closure. The link invariants we obtain are twisted $L^2$-Alexander torsions of the braid closure, and recover more topological information, such as the hyperbolic volumes of Dehn fillings. 
The proof needs us to first generalize several fundamental formulas for $L^2$-torsions, which have their own independent interest.

We then discuss how likely we are to generalize this process to yet more groups. In particular, a detailed study of the influence of Markov moves on $L^2$-Burau maps and two explicit counter-examples to Markov invariance suggest that twisted $L^2$-Alexander torsions of links are  the only link invariants we can hope to build from $L^2$-Burau maps with the present approach.
\end{abstract}

\tableofcontents

\section{Introduction}

The $L^2$-Burau maps and reduced $L^2$-Burau maps of braids were introduced in 2016 in \cite{BAC} by A. Conway and the author. These maps generalize the Burau representation of braid groups (introduced in \cite{Bu}) and are indexed by a positive real number $t$ and a group epimorphism $\gamma$ starting at a free group $\F_n$ of finite rank. In a sense, all the $L^2$-Burau maps are contained between the Burau representation and the Artin injection of the braid group $B_n$ in the automorphism group $\mathrm{Aut}(\F_n)$ of the free group $\F_n$. Moreover, A. Conway and the author proved in \cite{BAC} that for a braid $\beta \in B_n$, its image by the reduced $L^2$-Burau map $\overline{\mathcal{B}}^{(2)}_{t,\gamma_\beta}$ associated to the epimorphism $\gamma_\beta\colon \F_n \to G_{\beta} \cong \pi_1\left ( S^3 \setminus \hat{\beta}\right )$ 
yields the \textit{$L^2$-Alexander torsion} of the braid closure $\hat{\beta}$.
The $L^2$-Alexander torsion is an invariant of links introduced by Li--Zhang and Dubois--Friedl--Lück \cite{LZ, DFL}, whose construction can be compared with those of the twisted Alexander polynomials, and  which detects various topological information. As a consequence of the main result of \cite{BAC}, the reduced $L^2$-Burau map $\overline{\mathcal{B}}^{(2)}_{t,\gamma_\beta}$ thus
contains deep topological information about $\beta$ such as the hyperbolic volume or the genus of the link $\hat{\beta}$.

It is then natural to wonder whether $L^2$-Burau maps associated to other epimorphisms can similarly provide link invariants and detect topological information of the braid, and this article provides a partial positive answer to this question.

In Section \ref{sec:markov:adm}, we introduce the notion of \textit{Markov-admissibility} of a family $\QQ$ of group epimorphisms $Q_{\beta}\colon \F_{n(\beta)} \twoheadrightarrow G_{Q_{\beta}}$ indexed by braids $\beta \in \sqcup_{n\geqslant 1} B_n$. Roughly speaking, we say that such a family is Markov-admissible when for any two braids $\alpha,\beta$ that have isotopic closures (and thus are related by Markov moves), the associated epimorphisms $Q_{\alpha}$ and $Q_{\beta}$ descend ``to the same depth'' and are related by a sequence of commutative diagrams. Markov admissibility appears to be a necessary condition in order to construct Markov functions and link invariants from general families of epimorphisms indexed by braids.

In this paper we focus on a specific candidate for being a Markov function, namely the function
$$
F_{\mathcal{Q}}:=
\begin{pmatrix}
&\sqcup_{n\geqslant 1} B_n &\to& \mathcal{F}(\R_{>0}, \R_{>0})/\{t \mapsto t^m, m \in \Z \} \\
&\beta &\mapsto & 
\left [t \mapsto 
\dfrac{
	\det^r_{G_{Q_\beta}}\left (
	\overline{\mathcal{B}}^{(2)}_{t,Q_{\beta}}
	(\beta) - \Id^{\oplus (n-1)}
	\right )
}{\max(1,t)^n}
\right ]
\end{pmatrix},
$$
where $\QQ$ is a Markov-admissible family of epimorphisms and $\det^r_{G}$ is the \textit{regular Fuglede--Kadison determinant} for the group $G$, a  version of the determinant for infinite-dimensional $G$-equivariant operators on $\ell^2(G)$ such as the $L^2$-Burau maps (see Section \ref{sec:prelim} for a definition).

The first main result of this article is the following theorem (stated here without technical details for readability):
\begin{theo}[Theorem \ref{thm:burau:alexander:twisted}, Propositions \ref{prop:first} and \ref{prop:second}]\label{thm:intro:invariance}
	Let $\QQ$ be a Markov-admissible family of epimorphisms that descends to the groups of the braid closures or deeper. 
	Then $F_{\mathcal{Q}}$
	is a Markov function, and thus defines an invariant of links.
		Moreover, this link invariant is a twisted $L^2$-Alexander torsion.
\end{theo}

Part of Theorem \ref{thm:intro:invariance} was already proven in \cite{BAC},  without discussing Markov invariance. Indeed, when $\QQ$ is the family that  descends to the groups of the braid closures $G_\beta$, \cite[Theorem 4.9]{BAC} directly linked $F_\QQ$ to the $L^2$-Alexander torsions of the braid closures, as a variant of the well-known Alexander--Burau formula \cite{Bu}; and since the $L^2$-Alexander torsions were already known link invariants, Markov invariance was thus reciprocally guaranteed, and carefully studying Markov moves was unnecessary. We generalize this approach to \textit{twisted} $L^2$-Alexander torsions in Theorem \ref{thm:burau:alexander:twisted}; to do this, we first need to establish slight generalizations of several formulas for $L^2$-torsions, in Section \ref{sec:appendix}.

Propositions \ref{prop:first} and \ref{prop:second} provide another way of proving that the twisted $L^2$-Alexander torsions are invariants of links, by studying how Markov moves on braids modify reduced $L^2$-Burau maps and by using properties of the Fuglede--Kadison determinant. Although their conclusions are redundant with the one of Theorem \ref{thm:burau:alexander:twisted}, these two propositions have independent interest for future potential constructions of link invariants from $L^2$-Burau maps.

Theorem \ref{thm:intro:invariance} is not an equivalence in the sense that $F_\QQ$ could theoretically be a Markov function for other families $\QQ$, but the specific convenient cancellations that occur in matrix coefficients in the proofs of 
Propositions \ref{prop:first} and \ref{prop:second}
 make this seem unlikely. As further evidence, we present two families $\QQ$ such that $F_\QQ$ is not a Markov function, in the second main result of this article:
\begin{theo}[Theorems \ref{thm:contrex:abel} and \ref{thm:contrex:id}]\label{thm:intro:counter}
	Let $\QQ$ be either the family of identities of the free groups or the family of abelianizations of the free groups.
	Then $F_{\mathcal{Q}}$
	is not a Markov function.
\end{theo}
Our main tools to prove Theorem \ref{thm:intro:counter} are relations between Fuglede--Kadison determinants (which are technical to define and difficult to compute), Mahler measures of polynomials (notably studied by Boyd \cite{Bo}) and combinatorics on Cayley graphs of free groups (see \cite{BA6}).

As the reader will probably agree, the initial question (how to build link invariants from $L^2$-Burau maps) is still far from answered. We restricted ourselves to studying an intuitive form of a potential Markov function, namely $F_\QQ$, and we found that twisted $L^2$-Alexander torsions appeared to be the best link invariants we could obtain with it. This last point is unsurprising considering the form of $F_\QQ$ and its natural connection with the Alexander polynomial and its variations.

However, there may well be new link invariants to discover via other functions of the $L^2$-Burau maps, and we hope that our computations of how these maps change under Markov moves can be of use for future research in this vein.

The present article arose as a part of a wider project in collaboration with C. Anghel aiming to construct new knot invariants from $L^2$-versions of the Burau and Lawrence representations of braid groups. To attain this end, studying the influence of Markov moves on $L^2$-Burau maps appears to be a natural step.
Moreover, at the time of writing, the current paper is an expansion of the first half of the preprint \cite{BA56}. The second half of \cite{BA56} has since been expanded into the separate paper \cite{BA6}.

The article is organized as follows: in Section \ref{sec:prelim} we recall preliminaries on braid groups and $L^2$-invariants; 
 in Section \ref{sec:appendix}, we cover some improvements of classical properties of $L^2$-torsions;
 in Section \ref{sec:twisted}, we state and prove Theorem \ref{thm:burau:alexander:twisted};
finally, in Section \ref{sec:other} we introduce the notion of a Markov-admissible family of group epimorphisms and we study several examples and counter-examples of Markov invariance.

\section{Preliminaries}\label{sec:prelim}

In this section, we will set some notation and recall some fundamental properties.
We will mostly follow the conventions of \cite{BAC} and \cite{Lu}. Additional details on braid groups can be found in \cite{Bi, KT}.

\subsection{Braid groups}\label{subsec:braid}

The braid group~$B_n$ can  be seen as the set of  isotopy classes of orientation-preserving homeomorphisms of the punctured disk $D_n :=D^2 \setminus \lbrace p_1,..., p_n \rbrace$ which fix the boundary pointwise. Recall that~$B_n$ admits a presentation with $n-1$ generators~$\sigma_1,\sigma_2, \dots, \sigma_{n-1}$ following the relations~$\sigma_i \sigma_{i+1} \sigma_i=\sigma_{i+1} \sigma_i \sigma_{i+1}$ for each~$i$, and~$\sigma_i \sigma_j = \sigma_j \sigma_i$ if~$|i-j|>2$. Topologically, the generator~$\sigma_i$ is the braid whose~$i$-th component passes over the~$(i+1)$-th component.

\begin{figure}[!h]
\begin{tikzpicture}
\begin{scope}[xshift=0cm,scale=1]

\draw[very thick, fill=gray!25] (0,0) circle (2);
\draw[very thick, fill=white] (-1,0) circle (0.2);
\draw[very thick, fill=white] (0,0) circle (0.2);
\draw[very thick, fill=white] (1,0) circle (0.2);

\draw (0,2.2) node {$z$};
\draw (-1,-.9) node {$x_1$};
\draw (0,-.9) node {$x_2$};
\draw (1,-.9) node {$x_3$};

\draw[thick, ->] (0,2) 	..controls +(-1,-.2) and +(-.8,0).. (-1,-0.5);
\draw[thick] (0,2) 	..controls +(-1,-.2) and +(.8,0).. (-1,-0.5);

\draw[thick, ->] (0,2) 	..controls +(-.2,-.2) and +(-.8,0).. (0,-0.5);
\draw[thick] (0,2) 	..controls +(.2,-.2) and +(.8,0).. (0,-0.5);

\draw[thick, ->] (0,2) 	..controls +(1,-.2) and +(-.8,0).. (1,-0.5);
\draw[thick] (0,2) 	..controls +(1,-.2) and +(.8,0).. (1,-0.5);

\end{scope}

\begin{scope}[xshift=5cm,scale=1]

\draw[very thick, fill=gray!25] (0,0) circle (2);
\draw[very thick, fill=white] (-1,0) circle (0.2);
\draw[very thick, fill=white] (0,0) circle (0.2);
\draw[very thick, fill=white] (1,0) circle (0.2);

\draw (0,2.2) node {$z$};
\draw (-.5,-.5) node {$g_1$};
\draw (0.55,-.9) node {$g_2$};
\draw (.65,-1.6) node {$g_3$};

\draw[thick, ->] (0,2) 	..controls +(-1,-.2) and +(-.8,0).. (-1,-0.5);
\draw[thick] (0,2) 	..controls +(-1,-.2) and +(.8,0).. (-1,-0.5);

\draw[thick, ->] (0,2) 	..controls +(-1.5,-.2) and +(-2.5,0).. (0,-1);
\draw[thick] (0,2) 	..controls +(.8,-.2) and +(.8,0).. (0,-1);

\draw[thick,->] (0,2) arc (90:270:1.75);
\draw[thick] (0,2) arc (90:-90:1.75);

\end{scope}

\end{tikzpicture}
	\caption{Two families of loops on the punctured disk~$D_3$. }
	\label{fig:DiskTwisted}
\end{figure}
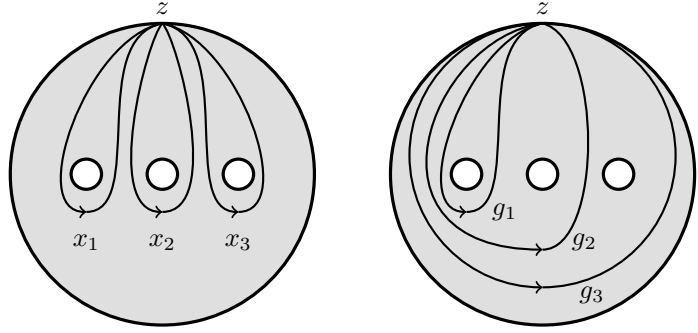

Fix a base point~$z$ of~$D_n$ and denote by~$x_i$ the simple loop based at~$z$ turning once around~$p_i$ counterclockwise for~$i=1,2,\dots, n$ (see Figure \ref{fig:DiskTwisted}). The group~$\pi_1(D_n,z)$ can then be identified with the free group~$\F_n$ on the~$x_i.$ 
If~$H_\beta$ is a homeomorphism of~$D_n$ representing a braid~$\beta$, then the induced automorphism~$h_{\beta}$ of the free group~$\F_n$  depends only on~$\beta$. It follows from the way we compose braids that~$h_{\alpha \beta}=h_{\beta}\circ h_{\alpha}$, and the resulting \textit{right} action of~$B_n$ on~$\F_n$ (named the \textit{Artin action}) can be explicitly described by
\[
h_{\sigma_i}(x_j)=
\begin{cases}
x_i x_{i+1} x_i^{-1}  & \mbox{if }  j=i, \\
x_i                            & \mbox{if }  j=i+1, \\ 
x_j                             & \mbox{otherwise, } \\ 
\end{cases} 
\hspace{1cm}
h_{\sigma_i^{-1}}(x_j)=
\begin{cases}
 x_{i+1}   & \mbox{if }  j=i, \\
x_{i+1}^{-1} x_i  x_{i+1}                          & \mbox{if }  j=i+1, \\ 
x_j                             & \mbox{otherwise. } \\ 
\end{cases} 
\] 

In this paper we will also use a second set of generators of $\F_n$, namely
$$g_1 := x_1, \  g_2 := x_1 x_2, \ \ldots, \ g_n := x_1 \ldots x_n.$$
Looking at Figure \ref{fig:DiskTwisted}, $g_i$ represents the class of the loop that circles the first $i$ punctures. On these generators, $B_n$ acts in the following way:
\[
h_{\sigma_i}(g_j)=
\begin{cases}
g_{i+1} g_i^{-1} g_{i-1}  & \mbox{if }  j=i, \\
g_j                             & \mbox{otherwise, } \\ 
\end{cases} 
\hspace{1cm}
h_{\sigma_i^{-1}}(g_j)=
\begin{cases}
g_{i-1} g_i^{-1} g_{i+1}   & \mbox{if }  j=i, \\
g_j                             & \mbox{otherwise,} \\ 
\end{cases} 
\] 
where we use the convention $g_0:=1$.

\subsection{Fox calculus}

Denoting by~$\F_n$ the free group on~$x_1,x_2,\dots,x_n$, and for $i \in \{1,\ldots,n\}$, the $i$-th \textit{Fox derivative} 
$\dfrac{\partial}{\partial x_i}: \mathbb{Z}[\F_n] \rightarrow \mathbb{Z}[\F_n]~$ 
(first introduced in \cite{Fox})
 is the linear extension of the map defined on $\F_n$ by:
 $\forall i,j \in \{1,\ldots,n\}, \forall u,v \in \F_n,$
$$
\frac{\partial x_j}{\partial x_i}=\delta_{i,j}, \ \ \ \ \ \ \ \ \ \frac{\partial x_j^{-1}}{\partial x_i}=-\delta_{i,j} x_j^{-1}, \ \ \ \ \ \ \ \ \  \frac{\partial (uv)}{\partial x_i}=\frac{\partial u}{\partial x_i}+u\frac{\partial v}{\partial x_i}.
$$

The following formula is often called the \textit{fundamental formula of Fox calculus}:

\begin{prop}\label{prop:fox}
Let $u \in \Z \F_n$ and $\epsilon\colon \Z \F_n \to \Z$ the ring epimorphism defined by $\epsilon\colon x_i \mapsto 1$ for all $i \in \{1,\ldots,n\}$. Then:
$$
u - \epsilon(u) \cdot  1 = \sum_{i=1}^n\dfrac{\partial u}{\partial x_i}  \cdot  (x_i -1)
.$$
\end{prop}

\subsection{Fuglede--Kadison determinant}

In this section we will give  short definitions of the von Neumann trace and the Fuglede--Kadison determinant.
More details can be found in \cite{Lu} and \cite{BAC}.

Let $G$ be a finitely generated group.
The Hilbert space $\ell^2(G)$ is the completion of the group algebra $\C G$, and the space of bounded operators on it is denoted $B(\ell^2(G))$. We will focus on \textit{right-multiplication operators} $R_w \in B(\ell^2(G))$, where $R_{\cdot}$ denotes the right regular action of $G$ on $\ell^2(G)$ extended to the group ring $\C G$ (and further extended to the rings of matrices $M_{p,q}(\C G)$).

For any element $w = a_0 \cdot 1_G + a_{1} g_1 \ldots + a_{r} g_r \in \C G$, the \textit{von Neumann trace} $\mathrm{tr}_G$ of the associated right multiplication operator is defined as
$$\mathrm{tr}_G(R_w) = \mathrm{tr}_G\left (
a_0 \Id_{\ell^2(G)} + a_{1} R_{g_1} \ldots + a_{r} R_{g_r} \right ) := a_0,$$
and the von Neumann trace for a finite square matrix over $\C G$ is given as the sum of the traces of the diagonal coefficients.

Now the most concise definition of the \textit{Fuglede--Kadison determinant} $\det_G(A)$ of a  right-multiplication operator $A$ is probably 
$$
\det{}_G(A) := 
\lim_{\varepsilon \to 0^+}
\left (\exp \circ \left (\dfrac{1}{2} \mathrm{tr}_G\right ) \circ  \ln \right )
\left (
(A_\perp)^*(A_\perp) + \varepsilon \Id \right )  \ \geqslant 0,$$
where $A_\perp$ is the restriction of $A$ to a supplementary of its kernel, $*$ is the adjunction and $\ln$ the logarithm of an operator in the sense of the holomorphic functional calculus. Compare with \cite[Theorem 3.14]{Lu} and Proposition \ref{prop:detFK:properties} below. We call the operator $A$ \textit{of determinant class} if $\det_G(A) \neq 0$.

The following properties concern the classical Fuglede--Kadison determinant $\det_G$ described in \cite[Chapter 3]{Lu}, which is not always multiplicative  if one deals with non injective operators. Moreover, this determinant forgets about the influence of the spectral value $0$, which surprisingly makes it take the value $1$ for the zero operator. More recent articles have used the \textit{regular Fuglede--Kadison determinant} $\det^r_G$ instead, which is defined for square injective operators, is zero for non injective operators, and is always multiplicative. In this paper we will work with both types of determinants, but the reader should be reassured that most of the statements we will make remain unchanged while replacing one determinant with the other (up to assumptions on injectivity usually). Similarly, the statements of the following Proposition \ref{prop:detFK:properties} admit immediate variants with $\det^r_G$. All statements of Proposition \ref{prop:detFK:properties} follow from \cite[Section 3]{Lu}, except for (6), which directly follows from the others.

\begin{prop}[\cite{Lu}]\label{prop:detFK:properties}
Let~$G$ be a countable discrete group and let $$A,B,C,D \in \sqcup_{p,q \in \N} R_{M_{p,q}(\C G)}$$ be general right multiplication operators. The Fuglede--Kadison determinant satisfies the following properties:
\begin{enumerate}
\item (multiplicativity) If $A, B$ are injective, square  and of the same size, then 
$$\det{}_G(A \circ B) = \det{}_G(A) \det{}_G(B).$$
\item (block triangular case) If $A, B$ are injective and square, then
$$\det{}_G
\begin{pmatrix}
A & C \\ 0& B
\end{pmatrix} 
= \det{}_G(A) \det{}_G(B)= \det{}_G
\begin{pmatrix}
A & 0 \\ D& B
\end{pmatrix},$$
where $C,D$ have the appropriate dimensions.
\item  (induction) If $\iota\colon G \hookrightarrow H$ is a group monomorphism, then 
$$\det{}_H (\iota (A)) 
= \det{}_G(A).$$
\item  (relation with the von Neumann trace) If $A$ is a positive operator, then
$$\det{}_G(A) = \left (\exp \circ \mathrm{tr}{}_G \circ \ln\right )(A).$$
\item (simple case) If~$g \in G$ is of infinite order, then for all~$t \in \C$ the operator $ \Id - t R_g$ is injective and~$$
\mathrm{det}_{G} ( \Id - t R_g) = \max ( 1 , |t|).$$
\item ($2\times 2$ trick) For all $A,B,C,D \in N(G)$ such
that $B$ is invertible, 
$\begin{pmatrix}
A & B \\ C& D
\end{pmatrix}$ is injective if and only if 
$D B^{-1} A - C$ is injective, and in this case
one has:
$$\det{}_G
\begin{pmatrix}
A & B \\ C& D
\end{pmatrix}
= \det{}_G(B) \det{}_G(D B^{-1} A - C).$$
\item (relation with Mahler measure) Let $G= \Z^d$, and $P\in \C[X_1^{\pm 1},\ldots,X_d^{\pm 1}]$ denote the  Laurent polynomial associated to the operator $A \in R_{\C \Z^d}$. Then
$$
\mathrm{det}_{\Z^d} (A) = \mathcal{M}(P) = \exp\left (\dfrac{1}{(2 \pi)^d} \int_0^{2\pi} \ldots \int_0^{2\pi} \ln\left (\left |P(e^{i \theta_1}, \ldots, e^{i \theta_d})\right |\right )d\theta_1 \ldots d\theta_d \right ),$$
where $\mathcal{M}$ is the Mahler measure.
\item (limit of positive operators) If $A$ is injective, then
$$\det{}_{G}(A)=\det{}_{G}(A^*) = \lim_{\varepsilon \to 0^+}
\sqrt{\det{}_{G}(A^* A + \varepsilon \Id)}.$$
\item (dilations) Let $\lambda \in \C^*$. Then:
$$\det{}_G\left (\lambda \ \Id^{\oplus n}\right ) = |\lambda|^n.$$
\end{enumerate}
\end{prop}

\begin{exemple}[\cite{Bo}, Section 4]\label{ex:mahler:boyd}
	The two-variable polynomial $1+X+Y \in \C[X,Y]$ has Mahler measure 
	$\mathcal{M}(1+X+Y) = e^{
		\frac{1}{\pi}\Im \mathrm{Li_2}\left (e^{i \pi/3}\right )
	}
	= 1.38135...$
	
	Thus, it follows from Proposition \ref{prop:detFK:properties} (7) that for $\Z^2=\langle x,y|xy=yx\rangle$,
	$$\det{}_{\Z^2}(\Id + R_x+R_y)=\mathcal{M}(1+X+Y)
	= 1.38135...$$
\end{exemple}

Example  \ref{ex:mahler:boyd} will be used later in the proof of Theorem \ref{thm:contrex:abel}.

\begin{remark}\label{rem:atiyah}
If the group $G$ satisfies the \textit{strong Atiyah conjecture} (see \cite[Chapter 10]{Lu}), then the right multiplication operator by any non-zero element of $\C G$ is injective, which makes it convenient to apply some parts of Proposition \ref{prop:detFK:properties}. Note that free groups and free abelian groups satisfy the strong Atiyah conjecture.
\end{remark}

\subsection{$L^2$-Burau maps on braids}
The $L^2$-Burau maps of braids were defined and studied in \cite{BAC}. Before this, a specific $L^2$-Burau map had been introduced in \cite{HL}, where it was called a quantization of the Burau representation.

Let $n\in \N^*$, $t>0$,  let $\Phi_{n}\colon \F_n \twoheadrightarrow \Z$ denote the projection that sends all free generators to $1$, and let $\gamma\colon \F_n \twoheadrightarrow G$ denote an epimorphism such that $\Phi_n$ factors through $\gamma$. 
Let
$\kappa(t,\Phi_{n},\gamma): \Z\F_n \to \R G$ denote the ring homomorphism that sends $g \in \F_n$ to $t^{\Phi_n(g)} \gamma(g) \in \R G$.

Then, following \cite{BAC}, the associated \textit{$L^2$-Burau map} on $B_n$ is
$$\mathcal{B}^{(2)}_{t,\gamma}\colon B_n \ni \beta \ \mapsto \ 
R_{\kappa(t,\Phi_{n},\gamma)(J)} \in B\left (\ell^2(G)^{\oplus n}\right ),
$$
where $J= \left ( \dfrac{\partial h_\beta(x_j)}{\partial x_i}\right )_{1 \leqslant i,j \leqslant n}$ is the Fox jacobian of $h_{\beta}$ for the base of the $x_i$.

The \textit{reduced $L^2$-Burau map} on $B_n$ (associated to the same parameters $t,\gamma$) is
$$\overline{\mathcal{B}}^{(2)}_{t,\gamma}\colon B_n \ni \beta \ \mapsto \ 
R_{\kappa(t,\Phi_{n},\gamma)(J')} \in B\left (\ell^2(G)^{\oplus (n-1)}\right ),
$$
where $J'= \left ( \dfrac{\partial h_\beta(g_j)}{\partial g_i}\right )_{1 \leqslant i,j \leqslant n-1}$ is the Fox jacobian of $h_{\beta}$ for the base of the $g_i$.

In the remainder of this article we will focus on reduced $L^2$-Burau maps.

Observe that $L^2$-Burau maps (reduced or not) can also be defined as maps over a certain homology of a cover of the punctured disk (see \cite{BAC} for details). Although these homological definitions may be more natural and  useful for further generalizations, the current article will only use the previous definitions via Fox jacobians.

Let us now state an (anti-)multiplication formula for the reduced $L^2$-Burau maps. 

\begin{prop}[\cite{BAC}]\label{prop:L2burau:mult}
For any $n, t, \gamma$ as above and any  braids $\alpha, \beta \in B_n$, we have:
$$
\overline{\mathcal{B}}^{(2)}_{t,\gamma }(\alpha \beta) =
\overline{\mathcal{B}}^{(2)}_{t,\gamma }(\beta) \circ 
\overline{\mathcal{B}}^{(2)}_{t,\gamma \circ h_{\beta}}(\alpha).
$$
\end{prop}

Remark that the unreduced $L^2$-Burau maps satisfy an identical formula. 

\begin{remark}
As Proposition \ref{prop:L2burau:mult} illustrates, $L^2$-Burau maps are \textit{crossed homomorphisms} from braid groups to certain groups of operators on Hilbert spaces. In this sense, $L^2$-Burau maps could (should?) rather be called \textit{$L^2$-Burau (crossed) representations} of the braid groups. However, as the current paper is a direct sequel of \cite{BAC}, we will keep calling them  \textit{$L^2$-Burau maps} for the reader's convenience.
\end{remark}

It follows from Proposition \ref{prop:L2burau:mult} that a reduced $L^2$-Burau map can be computed for any braid via knowing the values on the generators $\sigma_i$ of the braid group. For the reader's convenience and since they will be used in the remainder of this article, we now provide the images of the generators $\sigma_i$:
\begin{align*}
\overline{\mathcal{B}}_{t,\gamma}^{(2)}(\sigma_1)
&=
\begin{pmatrix}
-tR_{\gamma(g_{2}g_1^{-1})}  & 0 \\
\Id & \Id
\end{pmatrix}
\oplus \Id^{\oplus (n-3) }, \\
\overline{\mathcal{B}}_{t,\gamma}^{(2)}(\sigma_i)
&=\Id^{\oplus(i-2)} \oplus 
\begin{pmatrix}
\Id & tR_{\gamma(g_{i+1}g_i^{-1})} & 0 \\
0 & -tR_{\gamma(g_{i+1}g_i^{-1})}  & 0 \\
0 & \Id & \Id
\end{pmatrix}
\oplus \Id^{\oplus(n-i-2)} \ for \ 1<i<n-1,\\
\overline{\mathcal{B}}_{t,\gamma}^{(2)}(\sigma_{n-1})
&=\Id^{\oplus (n-3)} \oplus 
\begin{pmatrix}
\Id & tR_{\gamma(g_{n}g_{n-1}^{-1})}  \\
0 & -tR_{\gamma(g_{n}g_{n-1}^{-1})}   \\
\end{pmatrix},
\end{align*}
and the images of the inverses $\sigma_i^{-1}$ of the generators:
\begin{align*}
\overline{\mathcal{B}}_{t,\gamma}^{(2)}(\sigma_1^{-1})
&=
\begin{pmatrix}
-\frac{1}{t} R_{\gamma(g_1^{-1})}  & 0 \\
\frac{1}{t} R_{\gamma(g_1^{-1})} & \Id
\end{pmatrix}
\oplus \Id^{\oplus (n-3) }, \\
\overline{\mathcal{B}}_{t,\gamma}^{(2)}(\sigma_i^{-1})
&=\Id^{\oplus(i-2)} \oplus 
\begin{pmatrix}
\Id & \Id & 0 \\
0 & - \frac{1}{t} R_{\gamma(g_{i-1}g_i^{-1})}  & 0 \\
0 & \frac{1}{t} R_{\gamma(g_{i-1}g_i^{-1})} & \Id
\end{pmatrix}
\oplus \Id^{\oplus(n-i-2)} \ for \ 1<i<n-1,\\
\overline{\mathcal{B}}_{t,\gamma}^{(2)}(\sigma_{n-1}^{-1})
&=\Id^{\oplus (n-3)} \oplus 
\begin{pmatrix}
\Id & \Id  \\
0 & -\frac{1}{t} R_{\gamma(g_{n-2}g_{n-1}^{-1})}   \\
\end{pmatrix}.
\end{align*}

\begin{remark}\label{rem:det:l2burau}
It follows from what precedes and from Proposition \ref{prop:detFK:properties} (2) and (9) that for all $t>0$, we have
$\det_G\left (\overline{\mathcal{B}}_{t,\gamma}^{(2)}(\sigma_i^{\pm 1}) \right )= t^{\pm 1}$.
\end{remark}

\subsection{$L^2$-torsions}

This section covers some necessary definitions to state the results in Section \ref{sec:appendix}, which in turn will be used to prove Theorem \ref{thm:burau:alexander:twisted}. We refer to \cite{Lu} and \cite{BA2} for more details.

A \textit{finitely generated Hilbert $\mathcal{N}(G)$-module} is an Hilbert space $V$ on which there is a left $G$-action by isometries, and 
such that there exists a positive integer $m$ and an embedding $\phi$ of $V$ into $\bigoplus_{i=1}^m \ell^2(G)$ (in this paper, such spaces $V$ will always be of the form $\ell^2(G)^{\oplus n}$ for $n\in \N$). 

For $U$ and $V$ two finitely generated Hilbert $\mathcal{N}(G)$-modules, we will call $ f\colon U \rightarrow V$ a \textit{morphism of finitely generated Hilbert $\mathcal{N}(G)$-modules} if $f$ is a linear $G$-equivariant map, bounded for the respective scalar products of $U$ and $V$ (in this paper, these morphisms will simply be right multiplication operators).

A \textit{finite Hilbert $\NN(G)$-chain complex} $C_*$ is a sequence of morphisms of finitely generated Hilbert $\NN(G)$-modules
$$C_* = 0 \to C_n \overset{\partial_n}{\longrightarrow} C_{n-1} 
\overset{\partial_{n-1}}{\longrightarrow} \ldots
\overset{\partial_2}{\longrightarrow} C_1 \overset{\partial_1}{\longrightarrow} C_0 \to 0$$
such that $\partial_p \circ \partial_{p+1} =0$ for all $p$ (in this paper, $n$ will be at most $3$).

The \textit{$p$-th $L^2$-homology of $C_*$}  
$H_p^{(2)}(C_*) := Ker(\partial_p) / \overline{Im(\partial_{p+1})}$ is a finitely generated Hilbert $\NN(G)$-module. 
We say that $C_*$ is \textit{weakly acyclic} if its $L^2$-homology is trivial. We say that $C_*$ is of \textit{determinant class} if all the operators $\partial_p$ are of determinant class.

	Let $C_*$ be a finite Hilbert $\NN(G)$-chain complex as above. Its \textit{$L^2$-torsion} is
	$$T^{(2)}(C_*) := \prod_{i=1}^n \det {}_{\NN(G)}(\partial_i)^{(-1)^i} \in \R_{>0}$$
if $C_*$ is weakly acyclic and of determinant class, and is  $T^{(2)}(C_*):=0$ otherwise.

Let $\pi$ be a group and $\phi\colon \pi \to \Z$, $\gamma\colon \pi \to G$ two group homomorphisms.
We say that $(\pi,\phi,\gamma)$ \textit{forms an admissible triple} if $\phi \colon  \pi \to \Z$ factors 
through $\gamma$.
For $X$ a CW-complex, we say that $(X,\phi\colon \pi_1(X) \to \Z,\gamma\colon \pi_1(X) \to G)$ \textit{forms an admissible triple} if $(\pi_1(X),\phi,\gamma)$ forms one.
Let $(X,\phi,\gamma)$ be such an admissible triple, $\pi= \pi_1(X)$ and $t>0$. We define a ring homomorphism
$$\kappa(\pi, \phi, \gamma, t)\colon \begin{pmatrix}
& \Z[\pi]& \longrightarrow & \R[G] \\
& \sum_{j=1}^r m_j g_j & \longmapsto & \sum_{j=1}^r m_j t^{\phi(g_j)} \gamma(g_j)
\end{pmatrix}
$$
and we also denote $\kappa(\pi, \phi, \gamma, t)$ its induction over the $M_{p,q}(\Z[\pi])$.

Assume $X$ is compact. The cellular chain complex of $\widetilde{X}$ denoted
$C_*(\widetilde{X},\Z) = $ \\
$\left (\ldots \to \bigoplus_i \Z[\pi] \widetilde{e}_i^k \to \ldots\right )$
is a chain complex of left $\Z[\pi]$-modules.
Here the $\widetilde{e}_i^k$ are lifts of the cells $e_i^k$ of $X$.
The group $\pi$ acts on the right on $\ell^2(G)$ by $g \mapsto R_{\kappa(\pi, \phi, \gamma, t)(g)}$, an action which induces a structure of right $\Z[\pi]$-module on $\ell^2(G)$.
Let
$$C_*^{(2)}(X,\phi,\gamma,t) = \ell^2(G) \otimes_{\Z[\pi]}  C_*(\widetilde{X},\Z)$$
denote the finite Hilbert $\NN(G)$-chain complex obtained by tensor product via these left- and right-actions; we call $C_*^{(2)}(X,\phi,\gamma,t)$ \textit{a $\NN(G)$-cellular chain complex of $X$}.

	If $C_*^{(2)}(X,\phi,\gamma,t)$ is a $\NN(G)$-cellular chain complex of $X$, then denote
	$$ T^{(2)}(X,\phi,\gamma)(t) = T^{(2)}\left (C_*^{(2)}(X,\phi,\gamma,t)\right )$$
	the \textit{$L^2$-Alexander torsion of $(X,\phi,\gamma)$ at $t>0$}. It is non-zero if and only if $C_*^{(2)}(X,\phi,\gamma,t)$ is weakly acyclic and of determinant class.
	
For any two CW-complex structures $X, X'$ on a compact $3$-manifold $M$, and $\phi, \gamma$ as above, we have the following property (see for instance \cite{DFL}):
$$ \exists n \in \Z, \ \forall t>0, \ T^{(2)}(X,\phi,\gamma)(t) = t^n \cdot  T^{(2)}(X',\phi,\gamma)(t).$$
We therefore define the equivalence relation $\doteq$ on selfmaps of $\R_{\geqslant 0}$ by 
$$(t \mapsto f(t)) \doteq (t \mapsto g(t)) \ \Leftrightarrow \
\exists n \in \Z, \ \forall t>0, \ f(t)= t^n g(t),$$
and we sometimes abbreviate it $f(t)\doteq g(t)$. Hence we define the $L^2$-Alexander torsion of $M$ as the equivalence class
$T^{(2)}(M,\phi,\gamma)(t) := [T^{(2)}(X,\phi,\gamma)(t)]_{\doteq}$ for any CW-complex structure $X$ on  $M$, and we sometimes write 
$T^{(2)}(M,\phi,\gamma)(t) \doteq T^{(2)}(X,\phi,\gamma)(t)$ instead.
	
Let $L = L_1 \cup \ldots \cup L_c$ be a link in $S^3$, $M_L$ its exterior and $\alpha_L \colon G_L=\pi_1(M_L) \to \Z^c$ the abelianization of its group. Any homomorphism $\phi\colon G_L \to \Z$ factors through $\alpha_L$ and thus is written
$\phi = (n_1, \ldots,n_c) \circ\alpha_L$
where $n_1, \ldots,n_c \in \Z$.
Any admissible triple $(M_L,\phi,\gamma)$ can thus be written 
$(M_L,(n_1, \ldots,n_c) \circ\alpha_L,\gamma)$, and we will  denote
$$T^{(2)}_{L,(n_1, \ldots,n_c)}(\gamma)(t)
:= T^{(2)}(M_L,(n_1, \ldots,n_c) \circ\alpha_L,\gamma)(t)$$
the \textit{twisted $L^2$-Alexander torsion} associated to $L$,  to the coefficients $(n_1, \ldots,n_c)$, and to the morphism $\gamma$ (the \textit{twist}), at the value $t$. We  omit the \textit{twisted} when $\gamma=\id$.

Let us conclude this section by mentioning three important properties of $L^2$-torsions and $L^2$-Alexander torsions.

\begin{prop}\cite[Lemma 3.5]{BAC}\label{prop:irred:zero}
Let $M$ be a compact connected oriented $3$-manifold with empty or toroidal boundary. Let $\phi \in H^1(M,\Z)$. The following are equivalent:
\begin{enumerate}
\item $M$ is reducible (i.e. has a separating sphere),
\item $T^{(2)}(M)=0$,
\item $(t \mapsto T^{(2)}(M,\phi)(t))$ is the zero map.
\end{enumerate}
\end{prop}

\begin{theo}[Lück--Schick \cite{LS}]\label{thm:LS}
	Let $M$ be a compact connected oriented irreducible $3$-manifold with empty or toroidal boundary, and denote $\mathrm{vol}(M) \in \R_{\geqslant 0}$ its simplicial volume. Then
	$$T^{(2)}(M)=\exp\left (\dfrac{\mathrm{vol}(M)}{6\pi}\right ).$$
\end{theo}

\begin{theo}[Liu \cite{Liu}]\label{thm:liu}
	Let $M$ be a compact connected oriented irreducible $3$-manifold with empty or toroidal boundary, and let $\phi \in H^1(M,\Z)$. Then we can extract the Thurston norm $x_M(\phi)$ from the function $(t \mapsto T^{(2)}(M,\phi)(t))$ (as a difference of asymptotical degrees).
\end{theo}

\section{Some useful properties of $L^2$-torsions}\label{sec:appendix}

In this section we recall and generalize several natural properties of 
$L^2$-torsions, that will be used in the proof of Theorem \ref{thm:burau:alexander:twisted}.

The following Proposition \ref{prop:short} is a rephrasing of several results in \cite{Lu}, which concern short exact sequences of finite Hilbert $\NN(G)$-chain complexes. 
Recall that $0\to C_* \overset{\eta_*}{\to} D_* \overset{\rho_*}{\to} E_* \to 0$ is a \textit{short exact sequence of finite Hilbert $\NN(G)$-chain complexes} if $0\to C_p \overset{\eta_p}{\to} D_p \overset{\rho_p}{\to} E_p \to 0$ is exact for every $p$ and if $\eta_*, \rho_*$ commute with the boundary operators of $C_*,D_*, E_*$.

\begin{prop}[\cite{Lu}]\label{prop:short}
	Let $0\to C_* \overset{\eta_*}{\to} D_* \overset{\rho_*}{\to} E_* \to 0$ be a short exact sequence of finite Hilbert $\NN(G)$-chain complexes, such that for every $p \in \Z$, $\eta_p$ and $\rho_p$ are of determinant class. Then the following hold:
	\begin{enumerate}
		\item If two among $C_*,D_*, E_*$ are weakly acyclic, then the third one is as well.
		\item If  $C_*,D_*, E_*$ are all weakly acyclic, and if two of them are of determinant class, then the third one is as well.
		\item If  either $C_*,D_*, E_*$ are all weakly acyclic and of determinant class, or if $D_*$ is not, then the $L^2$-torsions satisfy
		$$T^{(2)}(D_*) \cdot \left (\prod_{p \in \Z} \left (\dfrac{\det_G(\rho_p)}{\det_G(\eta_p)}\right )^{(-1)^p}\right )
		=T^{(2)}(C_*) \cdot T^{(2)}(E_*).$$
	\end{enumerate}
\end{prop}

\begin{proof}
	Let us prove (1). If two among $C_*,D_*, E_*$ are weakly acyclic, then the long weakly exact homology sequence of finite Hilbert $\NN(G)$-modules
	$$LHS_* =  \ \ldots \to H^{(2)}_{n+1}(E_*) \to H^{(2)}_{n}(C_*) \to  H^{(2)}_{n}(D_*) \to H^{(2)}_{n}(E_*) \to \ldots $$
	of \cite[Theorem 1.21]{Lu} is trivial, and thus $C_*,D_*, E_*$ are all weakly acyclic.
	
	Now, (2) and (3) follow from \cite[Theorem 3.35 (1)]{Lu}, the assumptions  on $\iota_*$, $\rho_*$, and the fact that $LHS_*$ is trivial (which implies that $LHS_*$ is of determinant class and that its $L^2$-torsion is equal to $1$).
\end{proof}

The following proposition is a slight generalization of \cite[Theorem 2.12]{BA}, and concerns the invariance under simple homotopy equivalence.

\begin{prop}\label{prop:simple}
	Let $f: X \to Y$ be a simple homotopy equivalence between two finite CW-complexes inducing the group isomorphism $f_\sharp\colon \pi_1(X) \to \pi_1(Y)$. The triple $(Y, \phi, \gamma)$ is an
	admissible triple if and only if $(X, \phi \circ f_{\sharp} , \gamma \circ f_{\sharp})$ is one. Moreover, for all $t>0$:
	\begin{enumerate}
		\item $C^{(2)}_*(X, \phi \circ f_{\sharp} , \gamma \circ f_{\sharp}, t)$ is weakly acyclic if and only if $C^{(2)}_*(Y, \phi, \gamma, t)$ is,
		\item $C^{(2)}_*(X, \phi \circ f_{\sharp} , \gamma \circ f_{\sharp}, t)$ is weakly acyclic and of determinant class if and only if $C^{(2)}_*(Y, \phi, \gamma, t)$ is,
		\item $T^{(2)}(X, \phi \circ f_{\sharp} , \gamma \circ f_{\sharp})(t) \ \doteq \ T^{(2)}(Y, \phi, \gamma)(t).$
	\end{enumerate}
\end{prop}

\begin{proof}
	The proof is almost exactly as the one of \cite[Theorem 2.12]{BA}: we study the case where $f$ is an elementary expansion, and we relate $C^{(2)}_*(X, \phi \circ f_{\sharp} , \gamma \circ f_{\sharp}, t)$ and $C^{(2)}_*(Y, \phi, \gamma, t)$ through an exact sequence. Then we apply Proposition \ref{prop:short}.
\end{proof}

The following Proposition \ref{prop:gluing} is a slight generalization of the gluing formula of \cite[Theorem 3.1]{BA} and \cite[Proposition 3.5]{BA2} for $L^2$-Alexander torsions (which only stated that (1) and (2) together imply (3)).

\begin{prop}\label{prop:gluing} 
	Let $X,A,B,V$ be finite CW-complexes such that  $X = A \cup B$ and $V= A \cap B$. We denote the various inclusions (which are assumed to be cellular) and their inductions on fundamental groups as in the following diagrams:
	
	\begin{tikzpicture}
	[description/.style={fill=white,inner sep=2pt}] 
	\matrix(a)[matrix of math nodes, row sep=2em, column sep=2.5em, text height=1.5ex, text depth=0.25ex] 
	{ & A\\ V & & X\\ & B\\}; 
	\path[->](a-2-1) edge node[below]{$I_B$} (a-3-2); 
	\path[->](a-2-1) edge node[above]{$I_A$} (a-1-2);  
	\path[->](a-1-2) edge node[above]{$J_A$} (a-2-3); 
	\path[->](a-3-2) edge node[below]{$J_B$} (a-2-3); 
	\path[->](a-2-1) edge node[above]{$I$} (a-2-3);  
	
	\begin{scope}[xshift=6cm,rotate=0,scale=1]
	[description/.style={fill=white,inner sep=2pt}] 
	\matrix(a)[matrix of math nodes, row sep=2em, column sep=2.5em, text height=1.5ex, text depth=0.25ex] 
	{ & \pi_1(A)\\ \pi_1(V) & & \pi_1(X) & G\\ & \pi_1(B) & & \Z\\}; 
	\path[->](a-2-1) edge node[below]{$i_B$} (a-3-2); 
	\path[->](a-2-1) edge node[above]{$i_A$} (a-1-2);  
	\path[->](a-1-2) edge node[above]{$j_A$} (a-2-3); 
	\path[->](a-3-2) edge node[below]{$j_B$} (a-2-3); 
	\path[->](a-2-1) edge node[above]{$i$} (a-2-3);
	\path[->](a-2-3) edge node[above]{$\gamma$} (a-2-4);
	\path[->](a-2-3) edge node[below]{$\phi$} (a-3-4);
	\path[->, dotted](a-2-4) edge  (a-3-4);
	\end{scope}
	\end{tikzpicture}
	
	Let $(\pi_1(X),\phi\colon \pi_1(X) \to \Z,\gamma\colon \pi_1(X) \to G)$ be an admissible triple, and $t>0$. 
	If any two of the following properties are satisfied,
	\begin{enumerate}
		\item $C^{(2)}_*(V,\phi \circ i, \gamma \circ i,t)$  is weakly acyclic (resp. weakly acyclic and of determinant class),
		\item $C^{(2)}_*(A,\phi \circ j_A, \gamma \circ j_A,t)$ and $C^{(2)}_*(B,\phi \circ j_B, \gamma \circ j_B,t)$ are weakly acyclic (resp. weakly acyclic and of determinant class),
		\item $C^{(2)}_*(X,\phi, \gamma,t)$ is weakly acyclic (resp. weakly acyclic and of determinant class),
	\end{enumerate}
	then the third property is satisfied as well, and we have
	$$T^{(2)}(X,\phi, \gamma)(t)
	\ \dot{=} \ \dfrac{T^{(2)}(A,\phi \circ j_A, \gamma \circ j_A)(t) \cdot 
		T^{(2)}(B,\phi \circ j_B, \gamma \circ j_B)(t)}{T^{(2)}(V,\phi \circ i, \gamma \circ i)(t)}.
	$$
\end{prop}

\begin{proof}
	The proof works similarly as the one of \cite[Theorem 3.1]{BA}. Let us now sketch the modified arguments.
	Let $V_*, X_*$ denote the finite Hilbert $\NN(G)$-chain complexes of properties (1) and (3), and $C_*$ the direct sum of the two finite Hilbert $\NN(G)$-chain complexes in property (2). Observe that property (2) can be rephrased as $C_*$ being weakly acyclic (resp. weakly acyclic and of determinant class).
	As explained in \cite[Theorem 3.1]{BA} (via classical arguments), we have an exact sequence of finite Hilbert $\NN(G)$-chain complexes $0\to V_* \to C_* \to X_* \to 0$, where the horizontal operators are of determinant class. The result then follows from Proposition \ref{prop:short} and the computation of Fuglede--Kadison determinants of the horizontal operators.
\end{proof}

The following Proposition \ref{prop:Torres} is a slight generalization of the $L^2$-Torres formula of \cite[Theorem 4.4]{BA2} (which only stated that (1) implies (2) instead of their equivalence).

\begin{prop}\label{prop:Torres}
	Let $L=L_1\cup\ldots \cup L_c$ be a $c$-component link, and $L'=L \cup L_{c+1}$ a $(c+1)$-component link  admitting $L$ as a sublink. Let $M_L, M_{L'}$ denote the exteriors of $L$ and $L'$. Let $Q\colon \pi_1(M_{L'}) \twoheadrightarrow \pi_1(M_L)$ denote the group epimorphism induced by removing the component $L_{c+1}$. Let $\lambda \in \pi_1(M_{L'})$ denote the class of a preferred longitude of $L_{c+1}$.
	
	Let $\phi: \pi_1(M_{L}) \to \Z$ and $\gamma: \pi_1(M_{L}) \to G$ be group homomorphisms such that $(\pi_1(M_{L}),\phi,\gamma)$ forms an admissible triple.
	We can write $\phi = (n_1, \ldots, n_{c}) \circ \alpha_{L}$ and thus 
	$\phi \circ Q = (n_1, \ldots, n_{c}, 0) \circ \alpha_{L'}$ for some non zero vector $(n_1, \ldots, n_{c}) \in \Z^{c}$.
	
	Assume that $(\gamma \circ Q)(\lambda)$ is of infinite order in $G$. Then for all $t>0$, the following are equivalent:
	\begin{enumerate}
		\item $C_*^{(2)}(M_{L'},(n_1, \ldots, n_{c}, 0) \circ \alpha_{L'},\gamma \circ Q)(t)$ is weakly acyclic (resp. weakly acyclic and of determinant class),
		\item $C_*^{(2)}(M_{L},(n_1, \ldots, n_{c}) \circ \alpha_{L},\gamma)(t)$ is weakly acyclic (resp. weakly acyclic and of determinant class).
	\end{enumerate}
	Moreover, we have
	$$T^{(2)}_{L,(n_1, \ldots, n_{c})}(\gamma)(t)  \ \dot{=} \ 
	\dfrac{T^{(2)}_{L',(n_1, \ldots, n_{c}, 0)}(\gamma \circ Q)(t)}
	{\max(1,t)^{|\mathrm{lk}(L_1,L_{c+1}) n_1 + \ldots + \mathrm{lk}(L_{c},L_{c+1}) n_{c}|}}.$$
\end{prop}

\begin{proof}
	The proof is similar to \cite[Section 4]{BA2}. Here we use the generalized gluing formula of Proposition \ref{prop:gluing} instead of the weaker version of \cite[Proposition 3.5]{BA2}.
\end{proof}

\section{$L^2$-Burau maps of braids and twisted $L^2$-Alexander torsions of links}\label{sec:twisted}

For each $n \geqslant 1$ and each braid $\beta \in B_n$, let us denote
\begin{itemize}
	\item  $h_\beta$ the (Artin) group automorphism on $\F_{n(\beta)}$,
	\item $\gamma_\beta\colon \F_{n(\beta)} \twoheadrightarrow G_{\beta}$ the quotient by all relations of the form $\star = h_{\beta}(\star)$,
	\item  $\hat{\beta}$ the closure of $\beta$, which is a link in $S^3$,
	\item  $G_{\hat{\beta}}=\pi_1\left (S^3\setminus \hat{\beta}\right )$ the group of the link $\hat{\beta}$,
	\item  $\Psi_\beta\colon  G_\beta \overset{\sim}{\to}  G_{\hat{\beta}}$ the classical isomorphism,
	\item  $\Phi_{n}\colon \F_n \twoheadrightarrow \Z$ the epimorphism which sends the $n$ generators to $1$.
\end{itemize}

We can now state the main result of this paper.

\begin{theo}\label{thm:burau:alexander:twisted}
	Let  $n \geqslant 1$ and $\beta \in B_n$.
	Let $\psi_\beta\colon G_\beta \twoheadrightarrow \Gamma_{\psi_{\beta}}$ denote an epimorphism such that $\Phi_n$ factors through $\psi_\beta \circ  \gamma_\beta$.  Then we have the following relation between functions of $t>0$:
	$$\dfrac{
		\det^r_{\Gamma_{\psi_\beta}}\left (
		\overline{\mathcal{B}}^{(2)}_{t,\psi_\beta \circ  \gamma_\beta}
		(\beta) - \Id^{\oplus (n-1)}
		\right )
	}{\max(1,t)^n} \doteq T^{(2)}_{\hat{\beta},(1,\ldots,1)}\left (\psi_\beta \circ (\Psi_\beta)^{-1}\right )(t).$$ 
	In other words, the $L^2$-Alexander torsion of the link $\hat{\beta}$ twisted by the epimorphism 
	$\psi_\beta \circ (\Psi_\beta)^{-1}$ can be recovered from a specific $L^2$-Burau map of the braid $\beta$.
\end{theo}

Remark that  the case when $\psi_{\beta}$ is an isomorphism in Theorem \ref{thm:burau:alexander:twisted} was covered by the main result of \cite{BAC}.

\begin{proof}[Proof of Theorem \ref{thm:burau:alexander:twisted}]
	Let $L$ be a link in $S^3$, of exterior $M_L = S^3 \setminus \nu L$. Let $n\geqslant 1$ and $\beta \in B_n$ such that $L = \hat{\beta}$. We will mostly follow the way of the proof of \cite[Theorem 4.9]{BAC}, except for the fact that $\psi_\beta$ is now a general epimorphism and not the identity. This is why we  need the generalized properties of Section \ref{sec:appendix}. We will skip over some details already covered in \cite{BAC}.
	
	Recall that 
	$$P = \langle g_1, \ldots, g_n | r_1:= h_\beta(g_1)g_1^{-1}, \ldots, r_{n-1} := h_\beta(g_{n-1})g_{n-1}^{-1}\rangle$$
	is a presentation of $G_{\beta}$, and also of $G_{L}$, through the isomorphism $\Psi_\beta: G_{\beta} \overset{\sim}{\to} G_L$. Let $W_P$ be the $2$-dimensional CW-complex constructed from $P$. Recall that $W_P$ has a single $0$-cell, one $1$-cell for each generator of $P$, and one $2$-cell for each relator of $P$, each $2$-cell being glued on the wedge of circles that is the $1$-skeleton following the word in the generators formed by the relator in question.
	
	Now denote $L' = L \cup C_{\beta}$, where $C_{\beta}$ is the boundary circle of $D_n$ not coming from one of the punctures, when drawing $L$ as the closure of  $\beta$. Then
	$$P' = \langle g_1, \ldots, g_n,y | r'_1:= h_\beta(g_1)y g_1^{-1}y^{-1}, \ldots, r'_n := h_\beta(g_{n})y g_{n}^{-1}y^{-1}\rangle$$
	is a presentation of $G_{L'}$, with $y$ a meridian of $C_{\beta}$.
	Let $W_{P'}$ be the $2$-dimensional CW-complex constructed from $P'$.
	Since $L'$ is not split, $W_{P'}$ and $M_{L'}$ are therefore $K(G_{L'},1)$ spaces, and since the Whitehead group of $G_{L'}$ is trivial, we have that $W_{P'}$ is simple homotopy equivalent to $M_{L'}$.
	
	To simplify notations, let us denote $G:= \Gamma_{\psi_{\beta}}$, $\psi := \psi_\beta \circ (\Psi_\beta)^{-1}:G_{L} \twoheadrightarrow G$ and $\phi_L:= (1,\ldots,1) \circ \alpha_L\colon G_L \twoheadrightarrow \Z$. Let $t>0$. Let $Q\colon \pi_1(M_{L'}) \twoheadrightarrow \pi_1(M_L)$ denote the group epimorphism induced by removing the component $C_{\beta}$ (a specific instance of Dehn filling on $M_{L'}$). 
	Let us also denote as follows the five finite Hilbert $\NN(G)$-chain complexes we will use in the proof:
	\begin{itemize}
		\item $E_* := C^{(2)}_*(M_L,\phi_L,\psi,t)$,
		\item $E'_* := C^{(2)}_*(M_{L'},\phi_L\circ Q,\psi \circ Q,t)$,
		\item 
		$W_* := C^{(2)}_*(W_P,\phi_L,\psi,t)=\\$
		$$\bigoplus_{j=1}^{n-1} \ell^2(G)\widetilde{r_j}  \overunderset{\partial_2}{\begin{pmatrix} 
	\overline{\mathcal{B}}^{(2)}_{t,\psi}(\beta) -\Id^{n-1} 
				\\ * 	\end{pmatrix}}{\xrightarrow{\hspace{2.5cm}}}
		\bigoplus_{i=1}^{n} \ell^2(G)\widetilde{g_i}  \overunderset{\partial_1}{R_{(t\psi(g_1)-1,\ldots,t^n\psi(g_n)-1)}}{\xrightarrow{\hspace{2.5cm}}}
		\ell^2(G),$$
		\item $W'_* := C^{(2)}_*(W_{P'},\phi_L\circ Q,\psi \circ Q,t)=$\\
		\begin{align*}
		&\bigoplus_{j=1}^{n} \ell^2(G)\widetilde{r'_j}  \overunderset{\partial_2}{\begin{pmatrix}
				\overline{\mathcal{B}}^{(2)}_{t,\psi}(\beta) -\Id^{n-1} & 0
				\\ *  & 0 \\
				tR_{\psi(h_{\beta}(g_1))}-\Id \ \ldots	& t^n R_{\psi(g_n)}-\Id 	\end{pmatrix}}{\xrightarrow{\hspace{5.5cm}}}\\
		&\bigoplus_{i=1}^{n} \ell^2(G)\widetilde{g_i} \ \oplus \ \ell^2(G) \widetilde{y}  \overunderset{\partial_1}{R_{(t\psi(g_1)-1,\ldots,t^n\psi(g_n)-1,0)}}{\xrightarrow{\hspace{2.5cm}}}
		\ell^2(G),
	\end{align*}
		\item $D_* := 0 \to \ell^2(G) \widetilde{r'_n} \overunderset{\partial_2}{\Id - t^n R_{\psi(g_n)}}{\xrightarrow{\hspace{2cm}}} \ell^2(G) \widetilde{y} \overunderset{\partial_1}{0}{\longrightarrow} 0 \to 0$.
	\end{itemize}
	The forms of $W_*$ and $W'_*$ follow from the fact that we can use Fox calculus to describe the boundary operators in cellular chain complexes associated to the universal covers $\widetilde{W_P}$ and $\widetilde{W'_P}$.

	Observe that we have $T^{(2)}(W_*)=\dfrac{
		\det^r_{\Gamma_{\psi_\beta}}\left (
		\overline{\mathcal{B}}^{(2)}_{t,\psi_\beta \circ  \gamma_\beta}
		(\beta) - \Id^{\oplus (n-1)}
		\right )
	}{\max(1,t)^n}$. 
	Moreover, we have $T^{(2)}(E_*) \doteq T^{(2)}_{L,(1,\ldots,1)}(\psi)(t)$ by definition.
	Thus we only need to prove that $T^{(2)}(W_*) \doteq T^{(2)}(E_*)$. Let us state the two following facts.
	
	\underline{Fact 1:}  For $\lambda$  the class of a preferred longitude of $C_{\beta}$ in $G_{L'}$,  $\psi(Q(\lambda))$ has infinite order in $G$.
	This  is a consequence of  $\phi_L$ factoring through $\psi$ and $\phi(Q(\lambda))=n \neq 0$.

	\underline{Fact 2:}  There exists 
	$$0\to W_* \overset{\eta_*}{\longrightarrow} W'_* \overset{\rho_*}{\longrightarrow} D_* \to 0,$$ 
	a short exact sequence of finite Hilbert $\NN(G)$-chain complexes, with $\eta_2,\eta_1$ the obvious inclusions, $\eta_0=0$, $\rho_2,\rho_1$ the obvious projections, and $\rho_0=\Id$.
	This Fact follows from comparing the cells of $W_P$
	and $W_{P'}$, and from the definitions of $L^2$-Burau maps with Fox calculus.

	We can now establish the following equalities of $L^2$-torsions: 
	$$ T^{(2)}(E_*)  \ \dot{=} \ \dfrac{T^{(2)}(E'_*)}{\max(1,t)^n}  \ \dot{=} \ \dfrac{T^{(2)}(W'_*)}{\max(1,t)^n}  \ \dot{=} \ T^{(2)}(W_*),$$
	where the first equality follows from Proposition \ref{prop:Torres} and Fact 1, the second one follows from Proposition \ref{prop:simple}, and the third one follows from Proposition \ref{prop:short}, Fact 2 and Proposition \ref{prop:detFK:properties} (5) and (8).
	
	Observe that if any one of  $E_*, E'_*, W'_*, W_*$ is not weakly acyclic (resp. is weakly acyclic but not of determinant class), then no one is (resp. they all are), and in this case their $L^2$-torsions are all $0$  (and the previous equalities still stand).	
	
	This conludes the proof.	
\end{proof}

\begin{remark}
	In the previous proof, the exact sequence in Fact 2 is reversed from the one in the proof of \cite[Theorem 4.9]{BAC},  the latter being incorrect. Our proof of Theorem \ref{thm:burau:alexander:twisted}, which generalizes the one of \cite[Theorem 4.9]{BAC}, can thus be considered as an erratum of this mistake.
\end{remark}

\begin{remark}
	The proof of Theorem \ref{thm:burau:alexander:twisted} can be shortened if $L$ is non split, directly by using the simple homotopy equivalence between $M_L$ and $W_P$, like in the proof of  \cite[Theorem 4.9]{BAC}.
\end{remark}

As a consequence of Theorem \ref{thm:burau:alexander:twisted}, for a given braid $\beta \in B_n$, we conclude that the $L^2$-Burau maps of $\beta$ at the level of the group $G_\beta$ detects various topological information (possibly up to going to a lower level by applying an epimorphism $\psi_{\beta}$):

\begin{corollary}\label{cor:topo:info}
The reduced $L^2$-Burau maps $\left (t>0 \mapsto \overline{\mathcal{B}}^{(2)}_{t,\gamma_\beta}(\beta)\right )$ of a braid $\beta \in B_n$ detect the following topological information:
\begin{enumerate}
\item the simplicial volumes of the link exterior $M_{\hat{\beta}}$ and all its Dehn fillings,
\item the genus of the link $\hat{\beta}$ and the  Thurston norms  of $0$-surgeries on the link exterior $M_{\hat{\beta}}$,
\item the splitness of the link $\hat{\beta}$ and the irreducibility of all  Dehn fillings of the link exterior $M_{\hat{\beta}}$.
\end{enumerate}
\end{corollary}

\begin{proof}
From Theorem \ref{thm:burau:alexander:twisted} and the general surgery formula of \cite[Proposition 4.2]{BA2} (which generalizes Proposition \ref{prop:Torres}), we conclude that the reduced $L^2$-Burau maps $\left (t>0 \mapsto \overline{\mathcal{B}}^{(2)}_{t,\gamma_\beta}(\beta)\right )$ of a given braid $\beta \in B_n$ can recover the $L^2$-Alexander torsion of the link exterior $M_{\hat{\beta}}$ (with coefficients $(1,\ldots,1)$) and certain $L^2$-Alexander torsions of Dehn fillings $N$ of $M_{\hat{\beta}}$. 

More precisely, let $L= \hat{\beta}$ and let $N$ be obtained by Dehn fillings on $M_L$, with the induced epimorphism $Q: G_{L} \twoheadrightarrow \pi_1(N)$. Then two cases can occur:
\begin{enumerate}
\item[(a)] Either  $(1,\ldots,1)\circ \alpha_L$ factors through $Q$, which happens exactly when $N$ is obtained by one or more \textit{$0$-surgeries} (i.e. Dehn fillings with slope zero),
\item[(b)] Or  $(1,\ldots,1)\circ \alpha_L$ does not factor through $Q$, which happens in most cases.
\end{enumerate}

In case (a), we can still define a $L^2$-Alexander torsion $T^{(2)}(N,\phi_N)(t)$ (with $\phi_N$ such that $(1,\ldots,1)\circ \alpha_L=\phi_N \circ Q$) for any $t>0$, and express it from the $L^2$-Burau map
$\overline{\mathcal{B}}^{(2)}_{t,Q \circ \gamma_\beta}(\beta)$ via Theorem \ref{thm:burau:alexander:twisted}. Thus, by Theorem \ref{thm:liu}, the Thurston norms of $(M_L,(1,\ldots,1)\circ \alpha_L)$ (i.e. the genus of $L$) and of the $(N,\phi_N)$ can also be recovered, which proves (2).

In cases (a) and (b), we can still define the $L^2$-Burau maps and the $L^2$-Alexander torsions for the specific case of $t=1$. The $L^2$-Alexander torsion of $M_L$ (resp. $N$) is then simply its \textit{$L^2$-torsion}, which is known to recover the simplicial volume by Theorem \ref{thm:LS}. This proves (1).

Again in cases (a) and (b), Proposition \ref{prop:irred:zero} (which is an immediate generalization of \cite[Lemma 3.5]{BAC}) establishes that a $3$-manifold is reducible if and only if its $L^2$-torsion is zero. Since $M_L$ is reducible if and only if $L$ is split, this proves (3).
\end{proof}

The classical Burau representation of $B_n$ is not faithful for $n \geqslant 5$. On the other hand, the 
highest $L^2$-Burau map (i.e. the Fox jacobian of the Artin action on the free group $\F_n$) is known to be injective.
A natural question is thus: 

\begin{question}\label{qu:lowest:map}
For a given braid $\beta \in B_n$, what are the lowest $L^2$-Burau maps that distinguish $\beta$ from the trivial braid?
\end{question}

Corollary \ref{cor:topo:info} (see also \cite[Corollary 4.11]{BAC}) suggests that the $L^2$-Burau maps at the level of the group of the braid closure are promising candidates for Question \ref{qu:lowest:map}, as they detect various kinds of strong topological information.

\section{How to find other link invariants}\label{sec:other}

In this section, we discuss how likely we are to generalize the process of Theorem \ref{thm:burau:alexander:twisted} in order to construct more invariants of links from $L^2$-Burau maps of braids.

 In Section \ref{sec:markov:adm} we introduce the notion of Markov-admissibility of a family of epimorphisms, a necessary property for hoping to construct link invariants; in Section \ref{sec:invariance} we study Markov moves on $L^2$-Burau maps; in Section \ref{sec:counter:ex} we find two examples of families that cannot yields link invariants. This suggests that, regarding constructions of link invariants from $L^2$-Burau maps, Theorem \ref{thm:burau:alexander:twisted} is the best we can hope for with currently known formulas.

However, different formulas with $L^2$-Burau maps might satisfy Markov invariance and provide new link invariants. We hope that the following detailed descriptions of the actions of Markov moves can help for such future quests for new link invariants.

\subsection{Markov admissibility of a family of epimorphisms}\label{sec:markov:adm}

In this section we introduce the notion of Markov-admissibility for a family of epimorphisms indexed by braids, and we discuss several  examples of such families.

For each $n \geqslant 1$ and each braid $\beta \in B_n$, let us recall and fix some notation:
\begin{itemize}
\item  $n(\beta):=n$ the number of strands,
\item  $h_\beta$ the (Artin) group automorphism on $\F_{n(\beta)}$ (recall that $\beta \mapsto h_{\beta}$ is anti-multiplicative),
\item $\gamma_\beta\colon \F_{n(\beta)} \twoheadrightarrow G_{\beta}$ the quotient by all relations of the form $\star = h_{\beta}(\star)$,
\item  $\hat{\beta}$ the closure of $\beta$, a link in $S^3$,
\item  $G_{\hat{\beta}}=\pi_1\left (S^3\setminus \hat{\beta}\right )$ the group of the link $\hat{\beta}$,
\item  $\Phi_{n}\colon \F_n \twoheadrightarrow \Z$ the epimorphism which sends the $n$ generators to $1$,
\item $\iota_n \colon \F_n \hookrightarrow \F_{n+1}$ the group inclusion sending the $n$ generators of $\F_n$ to the first $n$ generators of $\F_{n+1}$.
\end{itemize}

\begin{defi}
A family $\mathcal{Q}$ of group epimorphisms of the form
$$\mathcal{Q} =\left \{ Q_{\beta}\colon \F_{n(\beta)}\twoheadrightarrow G_{Q_{\beta}}  \ | \ \beta \in \sqcup_{n\geqslant 1} B_n \right \}
$$
is called \textit{Markov-admissible} if it satisfies the following conditions:
\begin{enumerate}
\item For all $n\geqslant 1$ and all $\alpha, \beta \in B_n$, there exists a group homomorphism \\
$\chi^\mathcal{Q}_{\beta,\alpha}\colon G_{Q_{\beta}} \to G_{Q_{\alpha^{-1}\beta \alpha}}$
such that $Q_{\alpha^{-1}\beta \alpha} \circ h_\alpha = \chi^\mathcal{Q}_{\beta,\alpha} \circ Q_\beta$. Observe that this  implies that each $\chi^\mathcal{Q}_{\beta,\alpha}$ is uniquely defined and is an isomorphism.
\item For all $n\geqslant 1$, $\beta \in  B_n$ and $\varepsilon \in \{\pm 1\}$, there exists a group monomorphism
$\sigma^\mathcal{Q}_{\beta,\varepsilon}\colon G_{Q_{\beta}}
\hookrightarrow
G_{Q_{\sigma_n^{\varepsilon}\beta }}$ such that 
$Q_{\sigma_n^{\varepsilon}\beta} \circ \iota_n = \sigma^\mathcal{Q}_{\beta,\varepsilon} \circ Q_\beta$. Observe that each $Q_{\sigma_n^{\varepsilon}\beta}$ is uniquely defined.
\end{enumerate}
These two conditions are illustrated in Figure \ref{fig:markov:adm}.
\end{defi}

\begin{figure}[!h]

\begin{tikzpicture}

\begin{scope}[xshift=0cm]

\node at (0, 0)    {$G_{Q_{\beta}}$};

\draw[->,color=blue] (0.5,0)--(2.1,0);
\node[color=blue] at (1.3,0.3)    {$\chi^\QQ_{\beta,\alpha}$};
\node[color=blue] at (1.3,-0.2)    {$\sim$};

\node at (3, 0)    {$G_{Q_{\alpha^{-1}\beta\alpha }}$};

\node at (0,2)    {$\F_n$};

\draw[->>] (0,1.7)--(0,0.4);
\node at (0.3,1)    {$Q_\beta$};

\draw[->] (0.3,2)--(2.6,2);
\node at (1.3,2.3)    {$h_{\alpha}$};
\node at (1.3,2-0.2)    {$\sim$};

\node at (3,2)    {$\F_n$};

\draw[->>] (3,1.7)--(3,0.4);
\node at (3.7,1)    {$Q_{\alpha^{-1}\beta \alpha}$};

\end{scope}


\begin{scope}[xshift=7cm]

\node at (0, 0)    {$G_{Q_{\beta}}$};

\node[color=blue] at (0.5, -0.01)    {$\lhook\joinrel$};
\draw[->,color=blue] (0.55,0)--(2.3,0);
\node[color=blue] at (1.3,0.3)    {$\sigma^\QQ_{\beta,\varepsilon}$};

\node at (3, 0)    {$G_{Q_{\sigma_n^{\varepsilon}\beta }}$};

\node at (0,2)    {$\F_{n}$};

\draw[->>] (0,1.7)--(0,0.4);
\node at (0.3,1)    {$Q_\beta$};

\draw[->] (0.45,2)--(2.4,2);
\node at (0.4, 2-0.01)    {$\lhook\joinrel$};
\node at (1.3,2.3)    {$\iota_n$};

\node at (3,2)    {$\F_{n+1}$};

\draw[->>] (3,1.7)--(3,0.4);
\node at (3.5,1)    {$Q_{\sigma_n^{\varepsilon}\beta }$};

\end{scope}

\end{tikzpicture}

\caption{Conditions for $\mathcal{Q}$ to be Markov-admissible}
\label{fig:markov:adm}
\end{figure}

Roughly speaking, the epimorphisms in a Markov-admissible family will be compatible in a way that lets us hope to compute ($L^2$-)knot invariants by using them. Note that less restrictive definitions may be preferred in the future if we find better ways of computing $L^2$-objects such as Fuglede--Kadison determinants.

We now present examples of Markov-admissible families, which are all displayed in Figure \ref{fig:epimorphisms} for clarity. Moreover, Figure \ref{fig:epimorphisms} summarizes the results of Sections \ref{sec:invariance} and \ref{sec:counter:ex} concerning Markov invariance ({\color{green} Y} standing for yes and {\color{red} N} for no).

\begin{figure}[!h]
	
	\begin{tikzpicture}
	
	\node at (0, 4)    {$\F_n$};
	
	\node at (0.3, 4.2)    {\color{red} N};
	
	\draw[->>] (0,3.7)--(0,0.3);
	\node at (0.3, 2)    {$\Phi_n$};
	
	\node at (0, 0)    {$\Z$};
	
	\node at (0.5, -0.1)    {\color{green} Y};
	
	\draw[->>] (0.3,3.7)--(1.7,2.3);
	\node at (1.2, 3.2)    {$\gamma_\beta$};
	
	\draw[->>,dashed] (1.7,1.7)--(0.3,0.3);
	
	\node at (2,2)    {$G_{\beta}$};
	
	\node at (2.4,2.4)    {\color{green} Y};
	
	\draw[->] (2.3,2)--(4.7,2);
	\node at (3.5,2.2)    {$\Psi_\beta$};
	\node at (3.5,2-0.2)    {$\sim$};
	
	\node at (5,2)    {$G_{\hat{\beta}}$};
	
	\node at (5.5, 2)    {\color{green} Y};
	
	\draw[->>] (2.3,1.9)--(3.6,1.3);
	\node at (2.8, 1.4)    {$\psi_\beta$};
	
	\draw[->>,dashed] (3.6,0.8)--(0.3,0.1);
	
	\node at (4,1)    {$\Gamma_{\psi_\beta}$};
	
	\node at (4.5, 1)    {\color{green} Y};
	
	\draw[->>] (-0.2, 3.8)--(-1.8,2.2);
	\node at (-1.2, 3.2)    {$\varphi_n$};
	
	\draw[->>,dashed] (-1.8,1.7)--(-0.3,0.3);
	
	\node at (-2,2)    {$\Z^n$};
	
	\node at (-2.5, 2)    {\color{red} N};
	
	\end{tikzpicture}
	\label{fig:epimorphisms}
	\caption{When does $\overline{\mathcal{B}}^{(2)}_{t,\gamma}$ yield a map on braids which is invariant under Markov moves, for 
		$\gamma\colon\F_n \twoheadrightarrow G$ ?}
\end{figure}

\begin{exemple}
The family of identity morphisms $\mathcal{Q} =\left \{ \id_{\F_{n(\beta)}} \right \}$ is Markov-admissible, with $\chi^\QQ_{\beta,\alpha}=\id_{\F_{n(\beta)}}$ and $\sigma^\QQ_{\beta,\varepsilon} = \iota_{n(\beta)}$.
\end{exemple}

\begin{exemple}
The family $\mathcal{Q} =\left \{ \Phi_{n(\beta)} \right \}$ is Markov-admissible, with ${\chi^\QQ_{\beta,\alpha}=\sigma^\QQ_{\beta,\varepsilon} = \id_\Z}$.
\end{exemple}

\begin{exemple}\label{ex:ab}
The family of abelianizations $\mathcal{Q} =\left \{ \varphi_{n(\beta)}\colon \F_{n(\beta)} \twoheadrightarrow \Z^{n(\beta)} \right \}$ is Markov-admissible, with
$\sigma^\QQ_{\beta,\varepsilon}$ the inclusion $\Z^{n(\beta)} \hookrightarrow \Z^{n(\beta)+1}$ induced by $\iota_{n(\beta)}$, and
 $\chi^\QQ_{\beta,\alpha}$ the permutation on the canonical generators of $\Z^n$ corresponding to the permutation of $n(\alpha)$ strands induced by $\alpha \in B_n$. 
\end{exemple}

The following proposition is an elementary result in group theory, but is stated for the reader's convenience.

\begin{prop}\label{prop:quotient}
Let $f\colon G \to H$ be a group homomorphism, and $N$ a normal subgroup of $G$. If 
$f$ is surjective and $Ker(f) \subset N$ (in particular if $f$ is an isomorphism),
then $f$ induces an isomorphism between $G/N$ and $H/f(N)$.
\end{prop}

Let us now consider epimorphisms that descend to the fundamental groups of the braid closure complements.

\begin{prop}\label{prop:gamma}
The family $\QQ=\{ \gamma_\beta\colon \F_{n(\beta)}\twoheadrightarrow G_{\beta} \}$ is Markov-admissible. Moreover the monomorphisms $\sigma^\mathcal{Q}_{\beta,\varepsilon}$ are isomorphisms.
\end{prop}

\begin{proof}
\underline{First step: Markov 1:}\\
Take $n \geqslant 1$ and $\alpha,\beta \in B_n$. Note that
\begin{align*}
 h_\alpha (Ker(\gamma_\beta)) 
&= h_{\alpha} \left (\ \langle \langle h_\beta(x)x^{-1}; x \in \F_n \rangle \rangle \ \right )
=  \langle \langle h_{\beta \alpha}(x)h_{\alpha}(x)^{-1}; x \in \F_n \rangle \rangle \\
&=  \langle \langle h_{\alpha^{-1} \beta \alpha}(y) y^{-1}; y \in \F_n \rangle \rangle
= Ker(\gamma_{\alpha^{-1} \beta \alpha}),
\end{align*}
thus it follows from Proposition \ref{prop:quotient} that
the isomorphism $h_\alpha\colon \F_n \overset{\sim}{\to} \F_n$ induces the required isomorphism $\chi^\QQ_{\beta,\alpha}\colon G_{\beta} \overset{\sim}{\to} G_{\alpha^{-1} \beta \alpha}$.

\

\underline{Second step: Markov 2:}\\
Take $n \geqslant 1$ and $\beta \in B_n$. Let $\beta_+=\sigma_n^{-1}\iota(\beta) \in B_{n+1}$.
First notice that
\begin{align*}
&h_{\beta_+}(x_j)x_j^{-1}= 
 h_{\iota(\beta)}(h_{\sigma_n^{-1}}(x_j))x_j^{-1} =
h_{\iota(\beta)}(x_j)x_j^{-1}  \ \ \text{for} \ 1\leqslant j\leqslant n-1, \\
& h_{\beta_+}(x_n)x_n^{-1}= 
h_{\iota(\beta)}(h_{\sigma_n^{-1}}(x_n))x_n^{-1} =
h_{\iota(\beta)}(x_{n+1})x_n^{-1} =
x_{n+1}x_n^{-1}, \\
& h_{\beta_+}(x_{n+1})x_{n+1}^{-1}= 
h_{\iota(\beta)}(h_{\sigma_n^{-1}}(x_{n+1}))x_{n+1}^{-1} =
h_{\iota(\beta)}(x_{n+1}^{-1}x_n x_{n+1})x_{n+1}^{-1} =
x_{n+1}^{-1} h_{\iota(\beta)}(x_n).
\end{align*}
Hence $Ker(\gamma_{\beta_+}) = \langle \langle \ x_{n+1}x_n^{-1} \ ; \
\iota_n(Ker(\gamma_{\beta})) \ \rangle \rangle$.

We can now define $\sigma^\QQ_{\beta,-1}:=\left (
G_{\beta} \ni [x]_{G_{\beta}} \mapsto [\iota_n(x)]_{G_{\beta_+}} \in G_{\beta_+}
\right )$, where $x \in \F_n$ and $[\cdot]_{G}$ is the quotient class in $G$.
Since $\iota_n(Ker(\gamma_{\beta}))  \subset Ker(\gamma_{\beta_+})$, then $\sigma^\QQ_{\beta,-1}$ is a well-defined group homomorphism.

Let us prove that $\sigma^\QQ_{\beta,-1}$ is surjective. Let $[y]_{G_{\beta_+}} \in G_{\beta_+}$, with $y \in \F_{n+1}$. Let $y'\in \iota_n(\F_n)$ be the word constructed from $y$ by replacing all letters $x_{n+1}$ with $x_n$. Hence $[y]_{G_{\beta_+}}=[y']_{G_{\beta_+}} \in Im(\sigma^\QQ_{\beta,-1})$ and $\sigma^\QQ_{\beta,-1}$ is surjective.

Let us prove that $\sigma^\QQ_{\beta,-1}$ is injective. Let $[\iota_n(x)]_{G_{\beta_+}} \in G_{\beta_+}$ (with $x\in \F_n$) be trivial. Then $\iota_n(x) \in Ker(\gamma_{\beta_+}) = \langle \langle \ x_{n+1}x_n^{-1} \ ; \
\iota_n(Ker(\gamma_{\beta})) \ \rangle \rangle$. Thus $\iota_n(x)$ is a product of conjugates (in $\F_{n+1}$) of terms $(x_{n+1}x_n^{-1})^{\pm 1}$ and/or terms in $\iota_n(Ker(\gamma_{\beta}))$. But since $\iota_n(x)$ is a free word without the letter $x_{n+1}$, we conclude that the conjugates of terms $(x_{n+1}x_n^{-1})^{\pm 1}$ in $\iota_n(x)$ cancel each other. Thus
$\iota_n(x) \in \iota_n(Ker(\gamma_{\beta}))$ and $[x]_{G_{\beta}}=1$. Hence 
$\sigma^\QQ_{\beta,-1}$ is injective.

\

\underline{Third step: Markov 2 again:}\\
Take $n \geqslant 1$, $\beta \in B_n$, and let $\beta_+=\sigma_n\iota(\beta) \in B_{n+1}$. The proof is similar as in the Second step, with the following differences:
\begin{align*}
& h_{\beta_+}(x_n)x_n^{-1}= 
h_{\iota(\beta)}(x_n) x_{n+1} \left (h_{\iota(\beta)}(x_n) \right )^{-1} x_n^{-1},
 \\
& h_{\beta_+}(x_{n+1})x_{n+1}^{-1}= 
h_{\iota(\beta)}(x_n) x_{n+1}^{-1},\\
& Ker(\gamma_{\beta_+}) = \langle \langle \ h_{\iota(\beta)}(x_n) x_{n+1}^{-1} \ ; \
\iota_n(Ker(\gamma_{\beta})) \ \rangle \rangle,
\end{align*}
and we replace the  $x_{n+1}$ with $h_{\iota(\beta)}(x_n)$ in the proof of the surjectivity of $\sigma^\QQ_{\beta,1}$.
\end{proof}

\begin{remark}
In the second step of the previous proof, one can alternatively prove that $\sigma^\QQ_{\beta,-1}$ is an isomorphism by observing that it corresponds to the sequence of  Tietze transformations going from the presentation
$$\langle x_1, \ldots,x_n|h_{\beta}(x_1)x_1^{-1}, \ldots, h_{\beta}(x_{n-1})x_{n-1}^{-1}, h_{\beta}(x_n)x_n^{-1} \rangle$$
 of $G_{\beta}$ to the presentation
\begin{align*}
&  \langle x_1, \ldots,x_n, x_{n+1}|h_{\iota(\beta)}(x_1)x_1^{-1}, \ldots, h_{\iota(\beta)}(x_{n-1})x_{n-1}^{-1}, x_{n+1}x_n^{-1}, x_{n+1}^{-1} h_{\iota(\beta)}(x_n) \rangle = \\
& \langle x_1, \ldots,x_n, x_{n+1}|h_{\beta_+}(x_1)x_1^{-1}, \ldots, h_{\beta_+}(x_{n-1})x_{n-1}^{-1}, h_{\beta_+}(x_n)x_n^{-1}, h_{\beta_+}(x_{n+1})x_{n+1}^{-1} \rangle 
\end{align*}
of $G_{\beta_+}$. 
\end{remark}

Finally, let us fix notations for epimorphisms that go lower than the  groups of the braid closures.

\begin{exemple}\label{ex:psi}
	 Let $\left \{\psi_\beta\colon G_{\beta}\twoheadrightarrow \Gamma_{\psi_\beta}\right \}_{\beta \in \sqcup_{n\geqslant 1} B_n}$ be a family of epimorphisms such that 
\begin{itemize}
\item 	$\Phi_{n(\beta)}$ factors through $ \psi_\beta \circ  \gamma_\beta$,
\item  $\chi^\QQ_{\beta,\alpha}\left ( Ker(\psi_\beta)\right ) = Ker(\psi_{\alpha^{-1} \beta \alpha})$,
\item  $\sigma^\QQ_{\beta,\varepsilon}\left ( Ker(\psi_\beta)\right ) = Ker(\psi_{\sigma_n^\varepsilon \iota(\beta)})$.
\end{itemize}	
Then it follows from Propositions \ref{prop:quotient} and \ref{prop:gamma} that
the family 
$$\QQ=\left \{ \psi_\beta \circ  \gamma_\beta\colon \F_{n(\beta)}\twoheadrightarrow \Gamma_{\psi_\beta} \right \}$$ is Markov-admissible.

Note that the first assumption (that $\Phi_{n(\beta)}$ factors through $ \psi_\beta \circ  \gamma_\beta$) is not necessary for $\QQ$ to be Markov-admissible, but is relevant in the previous Theorem \ref{thm:burau:alexander:twisted} in order to compare the function $F_{\QQ}$ of the next section with $L^2$-Alexander torsions (for which such a factoring property is assumed, see \cite{BAC}).
\end{exemple}

\subsection{A sufficient condition for Markov invariance}\label{sec:invariance}

For $\QQ=\{Q_\beta \}$ a Markov-admissible family, we define the function
$$
F_{\mathcal{Q}}:=
\begin{pmatrix}
 &\sqcup_{n\geqslant 1} B_n &\to& \mathcal{F}(\R_{>0}, \R_{\geqslant 0})/\{t \mapsto t^m, m \in \Z \} \\
&\beta &\mapsto & 
\left [t \mapsto 
\dfrac{
	\det^r_{G_{Q_\beta}}\left (
	\overline{\mathcal{B}}^{(2)}_{t,Q_{\beta}}
	(\beta) - \Id^{\oplus (n-1)}
	\right )
}{\max(1,t)^n}
\right ]
\end{pmatrix},
$$
taking values in equivalence classes $[t \mapsto f(t) ]$ of selfmaps of $\R_{>0}$, up to multiplication by monomials with integer exponents.

The following Propositions \ref{prop:first} and \ref{prop:second} show that $F_{\mathcal{Q}}$ is invariant under both Markov moves when $\mathcal{Q}$ descends to the groups of braid closures or lower. As a consequence, they provide a second separate proof of the fact that the twisted $L^2$-Alexander torsions of Theorem \ref{thm:burau:alexander:twisted} are invariants of links.

\begin{prop}\label{prop:first}
Let $\QQ=\{ \psi_\beta \circ  \gamma_\beta\colon \F_{n(\beta)}\twoheadrightarrow \Gamma_{\psi_\beta} \}$ be as in Example \ref{ex:psi}.
Then $F_{\mathcal{Q}}$ is invariant under the first Markov move.
\end{prop}

\begin{proof} Let $\QQ=\{  Q_\beta :=\psi_\beta \circ  \gamma_\beta\colon \F_{n(\beta)}\twoheadrightarrow \Gamma_{\psi_\beta} \}$.
Let $n \geqslant 1$ be an integer, $t>0$, and $\alpha, \beta \in B_n$.
We will prove that
$$
\det{}^r_{\Gamma_{\psi_{\alpha^{-1}\beta \alpha}}}\left (
\overline{\mathcal{B}}^{(2)}_{t,Q_{\alpha^{-1}\beta \alpha} }
(\alpha^{-1}\beta \alpha) - \Id^{\oplus (n-1)}
\right ) =
\det{}^r_{\Gamma_{\psi_\beta}}\left (
\overline{\mathcal{B}}^{(2)}_{t,Q_\beta}
(\beta) - \Id^{\oplus (n-1)}
\right ).
$$ 

Observe that for any epimorphism $\gamma\colon \F_n \to G$,  Proposition \ref{prop:L2burau:mult} implies that:
	$$
\Id^{\oplus (n-1)} =
\overline{\mathcal{B}}^{(2)}_{t,\gamma }(1) =
\overline{\mathcal{B}}^{(2)}_{t,\gamma }(\alpha \alpha^{-1}) =
\overline{\mathcal{B}}^{(2)}_{t,\gamma }(\alpha^{-1}) \circ 
\overline{\mathcal{B}}^{(2)}_{t,\gamma \circ h_{\alpha^{-1}}}(\alpha),
$$
thus $
\overline{\mathcal{B}}^{(2)}_{t,\gamma }(\alpha^{-1}) =
\left (\overline{\mathcal{B}}^{(2)}_{t,\gamma \circ h_{\alpha^{-1}}}(\alpha)\right )^{-1}
$.

Consequently, we have for any epimorphism $\gamma\colon \F_n \to G$:
\begin{align*}
\overline{\mathcal{B}}^{(2)}_{t,\gamma}(\alpha^{-1}\beta \alpha) &=
	\overline{\mathcal{B}}^{(2)}_{t,\gamma}(\alpha) \circ
	\overline{\mathcal{B}}^{(2)}_{t,\gamma \circ h_{\alpha}}(\beta) \circ
 \overline{\mathcal{B}}^{(2)}_{t,\gamma \circ h_{\beta \alpha}}(\alpha^{-1})\\
 &=
\overline{\mathcal{B}}^{(2)}_{t,\gamma}(\alpha) \circ
\overline{\mathcal{B}}^{(2)}_{t,\gamma \circ h_{\alpha}}(\beta) \circ
\left (\overline{\mathcal{B}}^{(2)}_{t,\gamma \circ h_{\alpha^{-1}\beta \alpha}}(\alpha)\right )^{-1}.
\end{align*}
Hence, for any epimorphism $\gamma\colon \F_n \to G$:
\begin{align*}
&\det{}^r_{G}\left (
\overline{\mathcal{B}}^{(2)}_{t,\gamma}
(\alpha^{-1}\beta \alpha) - \Id^{\oplus (n-1)}
\right ) \\
&= \det{}^r_{G}\left (
\overline{\mathcal{B}}^{(2)}_{t,\gamma}(\alpha) \circ
\overline{\mathcal{B}}^{(2)}_{t,\gamma \circ h_{\alpha}}(\beta) \circ
\left (\overline{\mathcal{B}}^{(2)}_{t,\gamma \circ h_{\alpha^{-1}\beta \alpha}}(\alpha)\right )^{-1} - \Id^{\oplus (n-1)}
\right )\\
&= \det{}^r_{G}\left (
\overline{\mathcal{B}}^{(2)}_{t,\gamma \circ h_{\alpha}}(\beta) 
- 
\left ( \overline{\mathcal{B}}^{(2)}_{t,\gamma}(\alpha)   \right )^{-1} \circ 
\overline{\mathcal{B}}^{(2)}_{t,\gamma \circ h_{\alpha^{-1}\beta \alpha}}(\alpha)
\right ),
\end{align*}
where the second equality follows from Remark \ref{rem:det:l2burau} and Proposition \ref{prop:detFK:properties} (1).

For $\gamma= Q_{\alpha^{-1}\beta \alpha}$, we thus have:
\begin{align*}
&\det{}^r_{\Gamma_{\psi_{\alpha^{-1}\beta \alpha}}}\left (
\overline{\mathcal{B}}^{(2)}_{t,Q_{\alpha^{-1}\beta \alpha} }
(\alpha^{-1}\beta \alpha) - \Id^{\oplus (n-1)}
\right ) \\
&= \det{}^r_{\Gamma_{\psi_{\alpha^{-1}\beta \alpha}}}\left (
\overline{\mathcal{B}}^{(2)}_{t,Q_{\alpha^{-1}\beta \alpha} \circ h_{\alpha}}(\beta) 
- 
\left ( \overline{\mathcal{B}}^{(2)}_{t,Q_{\alpha^{-1}\beta \alpha}}(\alpha)   \right )^{-1} \circ 
\overline{\mathcal{B}}^{(2)}_{t,Q_{\alpha^{-1}\beta \alpha} \circ h_{\alpha^{-1}\beta \alpha}}(\alpha)
\right )\\
&= \det{}^r_{\Gamma_{\psi_{\alpha^{-1}\beta \alpha}}}\left (
\overline{\mathcal{B}}^{(2)}_{t,  \chi^\mathcal{Q}_{\beta,\alpha} \circ Q_\beta}(\beta) 
- 
\left ( \overline{\mathcal{B}}^{(2)}_{t,Q_{\alpha^{-1}\beta \alpha}}(\alpha)   \right )^{-1} \circ 
\overline{\mathcal{B}}^{(2)}_{t,Q_{\alpha^{-1}\beta \alpha} \circ h_{\alpha^{-1}\beta \alpha}}(\alpha)
\right ).
\end{align*}
Now, since for  every braid $\sigma \in B_n$, $\gamma_\sigma \circ h_{\sigma} = \gamma_\sigma$ and $Q_\sigma = \psi_\sigma \circ \gamma_\sigma$, we obtain:
\begin{align*}
&\det{}^r_{\Gamma_{\psi_{\alpha^{-1}\beta \alpha}}}\left (
\overline{\mathcal{B}}^{(2)}_{t,Q_{\alpha^{-1}\beta \alpha} }
(\alpha^{-1}\beta \alpha) - \Id^{\oplus (n-1)}
\right ) \\
&= \det{}^r_{\Gamma_{\psi_{\alpha^{-1}\beta \alpha}}}\left (
\overline{\mathcal{B}}^{(2)}_{t,  \chi^\mathcal{Q}_{\beta,\alpha} \circ Q_\beta}(\beta) 
- 
\left ( \overline{\mathcal{B}}^{(2)}_{t,Q_{\alpha^{-1}\beta \alpha}}(\alpha)   \right )^{-1} \circ 
\overline{\mathcal{B}}^{(2)}_{t,Q_{\alpha^{-1}\beta \alpha} \circ h_{\alpha^{-1}\beta \alpha}}(\alpha)
\right )\\
&= \det{}^r_{\Gamma_{\psi_{\alpha^{-1}\beta \alpha}}}\left (
\overline{\mathcal{B}}^{(2)}_{t,  \chi^\mathcal{Q}_{\beta,\alpha} \circ Q_\beta}(\beta) 
- 
\left ( \overline{\mathcal{B}}^{(2)}_{t,Q_{\alpha^{-1}\beta \alpha}}(\alpha)   \right )^{-1} \circ 
\overline{\mathcal{B}}^{(2)}_{t,Q_{\alpha^{-1}\beta \alpha} }(\alpha)
\right )\\
&= \det{}^r_{\Gamma_{\psi_{\alpha^{-1}\beta \alpha}}}\left (
\overline{\mathcal{B}}^{(2)}_{t,  \chi^\mathcal{Q}_{\beta,\alpha} \circ Q_\beta}(\beta) 
- 
\Id^{\oplus (n-1)}
\right )\\
&= \det{}^r_{\Gamma_{\psi_\beta}}\left (
\overline{\mathcal{B}}^{(2)}_{t,  Q_\beta}(\beta) 
- 
\Id^{\oplus (n-1)}
\right ),
\end{align*}
where the last equality follows from Proposition \ref{prop:detFK:properties} (3).
\end{proof}

\begin{prop}\label{prop:second}
Let $\QQ=\{ \psi_\beta \circ  \gamma_\beta\colon \F_{n(\beta)}\twoheadrightarrow \Gamma_{\psi_\beta} \}$ be as in Example \ref{ex:psi}.
Then $F_{\mathcal{Q}}$ is invariant under the second Markov move.
\end{prop}

\begin{proof}
	For any braid $\beta$ let us denote $Q_\beta := \psi_\beta \circ  \gamma_\beta \colon \F_{n(\beta)} \twoheadrightarrow \Gamma_{\psi_{\beta}}$.
	
	\
	
	\underline{First step: negative crossing:}
	
	Let $n \in \N_{\geqslant 1}$, $\beta \in B_n$ and 
$\beta_- = \sigma_n^{-1} \iota(\beta) \in B_{n+1}$. Let $t>0$. We will prove that
$$
\dfrac{
	\det^r_{\Gamma_{\psi_{\beta_-}}}\left (
	\overline{\mathcal{B}}^{(2)}_{t,Q_{\beta_-}}
	(\beta_-) - \Id^{\oplus n}
	\right )
}{\max(1,t)^{n+1}} = \dfrac{1}{t} \cdot 
\dfrac{
	\det^r_{\Gamma_{\psi_\beta}}\left (
	\overline{\mathcal{B}}^{(2)}_{t,Q_{\beta}}
	(\beta) - \Id^{\oplus (n-1)}
	\right )
}{\max(1,t)^n}.
$$
In order to do this, we will compose $\overline{\mathcal{B}}^{(2)}_{t,Q_{\beta_-}}
(\beta_-) - \Id^{\oplus n}$ with two operators (one named $\mathfrak{G}$ on the left, one named $\mathfrak{D}$ on the right) so that we obtain a block triangular operator with the upper left block ``mostly'' equal to $\overline{\mathcal{B}}^{(2)}_{t,Q_{\beta}}
(\beta) - \Id^{\oplus (n-1)}$.

First, it follows from Proposition \ref{prop:L2burau:mult} that
$$ \overline{\mathcal{B}}^{(2)}_{t,Q_{\beta_-}}
(\beta_-) = 
\overline{\mathcal{B}}^{(2)}_{t,Q_{\beta_-}}
\left (\sigma_n^{-1} \iota(\beta)\right ) =
\overline{\mathcal{B}}^{(2)}_{t,Q_{\beta_-}}
\left (\iota(\beta)\right )
\overline{\mathcal{B}}^{(2)}_{t,Q_{\beta_-}\circ h_{\iota(\beta)}}
\left (\sigma_n^{-1} \right ).
$$

Recall that $\overline{\mathcal{B}}^{(2)}_{t,id}
\left (\sigma_n^{-1} \right ) = \begin{pmatrix}
\Id & \ & \ & \ & 0 \\
\ & \ddots & \ & \ & \vdots \\
\ & \ &\Id & \ & 0 \\
\ & \ & \ &\Id &\Id \\
0 & \ldots & 0 & 0 & -\frac{1}{t} R_{g_{n-1}g_n^{-1}}
\end{pmatrix}$, thus the operator defined as
$\mathfrak{D} := \left (\overline{\mathcal{B}}^{(2)}_{t,Q_{\beta_-}\circ h_{\iota(\beta)}}
\left (\sigma_n^{-1} \right )\right )^{-1}$ is equal to:
$$\mathfrak{D}:= \left (\overline{\mathcal{B}}^{(2)}_{t,Q_{\beta_-}\circ h_{\iota(\beta)}}
\left (\sigma_n^{-1} \right )\right )^{-1} = \begin{pmatrix}
\Id & \ & \ & \ & 0 \\
\ & \ddots & \ & \ & \vdots \\
\ & \ &\Id & \ & 0 \\
\ & \ & \ &\Id &  {t} R_{\left (Q_{\beta_-}\circ h_{\iota(\beta)}\right )(g_{n}g_{n-1}^{-1})} \\
0 & \ldots & 0 & 0 & -{t} R_{\left (Q_{\beta_-}\circ h_{\iota(\beta)}\right )(g_{n}g_{n-1}^{-1})}
\end{pmatrix}.$$
 We therefore compute:
\begin{align*}
\overline{\mathcal{B}}^{(2)}_{t,Q_{\beta_-}} 
(\beta_-) \circ  \mathfrak{D}&=
\overline{\mathcal{B}}^{(2)}_{t,Q_{\beta_-}}
\left (\iota(\beta)\right )
\circ 
\overline{\mathcal{B}}^{(2)}_{t,Q_{\beta_-}\circ h_{\iota(\beta)}}
\left (\sigma_n^{-1} \right ) \circ \mathfrak{D}\\
& =
\overline{\mathcal{B}}^{(2)}_{t,Q_{\beta_-}}
\left (\iota(\beta)\right )\\
&=
\begin{pmatrix}
\ & \ & \ & 0 \\
\ & \overline{\mathcal{B}}^{(2)}_{t,Q_{\beta_-}\circ\iota_{\F_n}}(\beta) & \ & \vdots \\
\ & \ & \ & 0 \\
 & \mathcal{R} &  &\Id 
\end{pmatrix},
\end{align*}
where the row $\mathcal{R}$ is the right multiplication operator by the row
$$  \begin{pmatrix}
\kappa(t,\Phi_{n+1},Q_{\beta_-})\left (\dfrac{\partial h_{\iota(\beta)}(g_j)}{\partial g_n}\right )
\end{pmatrix}
= \begin{pmatrix}
	\kappa(t,\Phi_{n+1}\circ \iota_{\F_n},Q_{\beta_-}\circ \iota_{\F_n})\left (\dfrac{\partial h_{\beta}(g_j)}{\partial g_n}\right )
\end{pmatrix},
$$
that has $n-1$ coefficients in $\R \Gamma_{\psi_{\beta_-}}$.

Now,  the fundamental formula of Fox calculus (Proposition \ref{prop:fox}) implies that:
$$
\begin{pmatrix}
R_{g_1-1} & \ldots & R_{g_n-1} 
\end{pmatrix} 
\cdot 
\begin{pmatrix}
R_{\dfrac{\partial h_{\iota(\beta)}(g_j)}{\partial g_i}}
\end{pmatrix}
_{1\leqslant i,j \leqslant n}
=\begin{pmatrix}
R_{h_{\iota(\beta)}(g_1)-1} & \ldots & R_{h_{\iota(\beta)}(g_n)-1} 
\end{pmatrix},
$$
thus, by applying $\kappa(t,\Phi_{n+1},Q_{\beta_-})$  and by definition of $\overline{\mathcal{B}}^{(2)}_{t,Q_{\beta_-}}
\left (\iota(\beta)\right )$, we obtain:
\begin{align*}
&\begin{pmatrix}
t R_{Q_{\beta_-}(g_1)}-\Id & \ldots & t^n R_{Q_{\beta_-}(g_n)}- \Id
\end{pmatrix} 
\cdot 
\overline{\mathcal{B}}^{(2)}_{t,Q_{\beta_-}}
\left (\iota(\beta)\right ) \\
&= \begin{pmatrix}
t R_{\left (Q_{\beta_-}\circ h_{\iota(\beta)}\right )(g_1)}-\Id & \hspace*{1.8cm} \ldots &
\hspace*{1.4cm} & t^n R_{\left (Q_{\beta_-}\circ h_{\iota(\beta)}\right )(g_n)}- \Id
\end{pmatrix}\\
&= \begin{pmatrix}
t R_{\left (Q_{\beta_-}\circ h_{\iota(\beta)}\right )(g_1)}-\Id & \ldots & 
t^{n-1} R_{\left (Q_{\beta_-}\circ h_{\iota(\beta)}\right )(g_{n-1})}-\Id &
t^n R_{Q_{\beta_-}(g_n)}- \Id
\end{pmatrix},
\end{align*}
where the second equality follows from the fact that $\beta \in B_{n}$ leaves $g_n$ unchanged in the Artin action.
Let us thus define the following block triangular operator $\mathfrak{G}$:
$$\mathfrak{G}:=
\begin{pmatrix}
 	 \Id & \ & \ & 0 \\
	\ & \ddots & \ & \vdots \\
	\ & \ & \Id & 0 \\
	t R_{Q_{\beta_-}(g_1)}-\Id & \ldots 
	& t^{n-1} R_{Q_{\beta_-}(g_{n-1})}- \Id
	& t^n R_{Q_{\beta_-}(g_n)}- \Id
\end{pmatrix}.$$
It immediately follows from what precedes that
\begin{align*}
& \mathfrak{G} \circ
\left (\overline{\mathcal{B}}^{(2)}_{t,Q_{\beta_-}} 
(\beta_-)\right ) \circ  \mathfrak{D} =
\mathfrak{G} \circ \left (\overline{\mathcal{B}}^{(2)}_{t,Q_{\beta_-}}
\left (\iota(\beta)\right )\right )=\\
&\begin{pmatrix}
\ & \ & \ & 0 \\
\ & \overline{\mathcal{B}}^{(2)}_{t,Q_{\beta_-}\circ\iota_{\F_n}}(\beta) & \ & \vdots \\
\ & \ & \ & 0 \\
t R_{\left (Q_{\beta_-}\circ h_{\iota(\beta)}\right )(g_1)}- \Id
\hspace*{-0.7cm}
 & \ldots  & 
 \hspace*{-0.7cm}
t^{n-1} R_{\left (Q_{\beta_-}\circ h_{\iota(\beta)}\right )(g_{n-1})}-\Id &
t^n R_{Q_{\beta_-}(g_n)}- \Id
\end{pmatrix}.
\end{align*}
On the other hand, we compute $\mathfrak{G} \circ
\left (\Id^{\oplus n} \right ) \circ  \mathfrak{D} =$
$$
\begin{pmatrix}
\Id & \  & \ & 0 \\
\ & \ddots  & \ & \vdots \\
\ & \  & \ & 0 \\
\ &  \ & \Id &  {t} R_{\left (Q_{\beta_-}\circ h_{\iota(\beta)}\right )(g_{n}g_{n-1}^{-1})} \\
t R_{Q_{\beta_-}(g_1)}- \Id
& 
\ldots  & 
t^{n-1} R_{Q_{\beta_-}(g_{n-1})}-\Id &
\star 
\end{pmatrix},
$$
with $\star = t^n R_{Q_{\beta_-}\left (g_n 
	h_{\iota(\beta)}(g_{n-1}^{-1}) g_{n-1}
	\right )}-t^{n+1} R_{Q_{\beta_-}\left (g_n 
	h_{\iota(\beta)}(g_{n-1}^{-1}) g_{n}
	\right )}
$.

Hence
\begin{align*}
& \mathfrak{G} \circ
\left (\overline{\mathcal{B}}^{(2)}_{t,Q_{\beta_-}} 
(\beta_-)-\Id^{\oplus n}\right ) \circ  \mathfrak{D} =\\
&\begin{pmatrix}
\ & \  & \ & 0 \\
\ & \overline{\mathcal{B}}^{(2)}_{t,Q_{\beta_-}\circ\iota_{\F_n}}(\beta) -\Id^{\oplus (n-1)}  & \ & \vdots \\
\ & \  & \ & 0 \\
\ &  \ & \ &  -{t} R_{\left (Q_{\beta_-}\circ h_{\iota(\beta)}\right )(g_{n}g_{n-1}^{-1})} \\
\ldots 
& 
t^{j} R_{\left (Q_{\beta_-}\circ h_{\iota(\beta)}\right )(g_{j})}-
t^{j} R_{Q_{\beta_-}(g_{j})}  & 
\ldots &
\square 
\end{pmatrix},
\end{align*}
where $j \in \{1, \ldots , n-1\}$ and
\begin{align*}
\square= t^n R_{Q_{\beta_-}(g_n)}-\Id -t^n R_{Q_{\beta_-}\left (g_n 
	h_{\iota(\beta)}(g_{n-1}^{-1}) g_{n-1}
	\right )}
+t^{n+1} R_{Q_{\beta_-}\left (g_n 
	h_{\iota(\beta)}(g_{n-1}^{-1}) g_{n}
	\right )}.
\end{align*}
Now we use the fact that our epimorphism $Q_{\beta_-}$ descends deeper than the braid closure group: indeed,
for every $j \in \{1, \ldots , n-1\}$, we have:
$$
Q_{\beta_-} (g_j) = \left (Q_{\beta_-} \circ h_{\beta_-}\right )(g_j)= \left (Q_{\beta_-} \circ h_{\iota(\beta)} \circ h_{\sigma_n^{-1}}\right )(g_j)= \left (Q_{\beta_-} \circ h_{\iota(\beta)} \right )(g_j).
$$
Thus $\mathfrak{G} \circ
\left (\overline{\mathcal{B}}^{(2)}_{t,Q_{\beta_-}} 
(\beta_-)-\Id^{\oplus n}\right ) \circ  \mathfrak{D}$ is actually upper block triangular and  equal to:
$$
\begin{pmatrix}
\ & 0 \\
 \overline{\mathcal{B}}^{(2)}_{t,Q_{\beta_-}\circ\iota_{\F_n}}(\beta) -\Id^{\oplus (n-1)}   & \vdots \\
\ & 0 \\
  \  &  -{t} R_{\left (Q_{\beta_-}\circ h_{\iota(\beta)}\right )(g_{n}g_{n-1}^{-1})} \\
0  &
-\Id 
+t^{n+1} R_{Q_{\beta_-}\left (g_n 
	h_{\iota(\beta)}(g_{n-1}^{-1}) g_{n}
	\right )} 
\end{pmatrix}.
$$
By applying $\det^r_{\Gamma_{\psi_{\beta_-}}}$ to the previous equality,
it therefore follows from Proposition \ref{prop:detFK:properties} (1), (2), (5), (9) and Remark \ref{rem:det:l2burau} that
\begin{align*}
&\max(1,t)^{n}  \cdot 
	\det{}^r_{\Gamma_{\psi_{\beta_-}}}\left (
	\overline{\mathcal{B}}^{(2)}_{t,Q_{\beta_-}}
	(\beta_-) - \Id^{\oplus n}
	\right )
\cdot t\\
&= 	\det{}^r_{\Gamma_{\psi_{\beta_-}}}\left (
	\overline{\mathcal{B}}^{(2)}_{t,Q_{\beta_-}\circ\iota_{\F_n}}
	(\beta) - \Id^{\oplus (n-1)}
	\right ) \cdot 
\max(1,t)^{n+1}.
\end{align*}
We conclude by using the fact that
$Q_{\beta_-}\circ\iota_{\F_n}= \sigma^{\mathcal{Q}}_{\beta,-1}  \circ Q_\beta$ (by
 Markov-admissibility of $\mathcal{Q}$) and Proposition \ref{prop:detFK:properties} (3).

	\
	
	\underline{Second step: positive crossing:}
			
Let $\beta_+ = \sigma_n \iota(\beta) \in B_{n+1}$. 
We will proceed almost exactly as in the first step, except for the following differences:
\begin{itemize}
\item We now aim to prove that
 $$\dfrac{
	\det{}^r_{\Gamma_{\psi_{\beta_+}}}\left (
	\overline{\mathcal{B}}^{(2)}_{t,Q_{\beta_+}}
	(\beta_+) - \Id^{\oplus n}
	\right )
}{\max(1,t)^{n+1}} = 
\dfrac{
	\det{}^r_{\Gamma_{\psi_\beta}}\left (
	\overline{\mathcal{B}}^{(2)}_{t,Q_{\beta}}
	(\beta) - \Id^{\oplus (n-1)}
	\right )
}{\max(1,t)^n}.$$
\item The operator $\mathfrak{D}$ becomes:
$$\mathfrak{D}:= \left (\overline{\mathcal{B}}^{(2)}_{t,Q_{\beta_+}\circ h_{\iota(\beta)}}
\left (\sigma_n \right )\right )^{-1} = \begin{pmatrix}
\Id & \ & \ & \ & 0 \\
\ & \ddots & \ & \ & \vdots \\
\ & \ & \Id & \ & 0 \\
\ & \ & \ & \Id &  \Id \\
0 & \ldots & 0 & 0 & -\frac{1}{t} R_{Q_{\beta_+}(g_{n}g_{n+1}^{-1})}
\end{pmatrix}.$$
\item The final lower right coefficient $\star$ of $\mathfrak{G} \circ  \mathfrak{D}$ becomes
$$\star= 
t^{n-1} R_{Q_{\beta_+}(g_{n-1})}
-\Id 
-t^{n-1} R_{Q_{\beta_+}\left (g_n 
	g_{n+1}^{-1} g_{n}
	\right )}
+\frac{1}{t} R_{Q_{\beta_+}\left (g_n 
	g_{n+1}^{-1}\right )}.$$
\item The final lower right coefficient $\square$ of $\mathfrak{G} \circ \left ( \overline{\mathcal{B}}^{(2)}_{t,Q_{\beta_+}} 
(\beta_+)-\Id^{\oplus n}\right )  \circ  \mathfrak{D}$ becomes
$$\square= 
-t^{n-1} R_{Q_{\beta_+}(g_{n-1})}
+t^n R_{Q_{\beta_+}(g_{n })}
+t^{n-1} R_{Q_{\beta_+}\left (g_n 
	g_{n+1}^{-1} g_{n}
	\right )}
-\frac{1}{t} R_{Q_{\beta_+}\left (g_n 
	g_{n+1}^{-1}\right )}.$$
\item The simplification 
$$\square= t^n R_{Q_{\beta_+}(g_{n })}
-\frac{1}{t} R_{Q_{\beta_+}\left (g_n 
	g_{n+1}^{-1}\right )}$$
comes from the fact that in the ring $\Z \F_{n+1}$ we have the equalities:
\begin{align*}
g_{n-1}- g_n g_{n+1}^{-1}g_{n}
&= g_{n} \left (	g_n^{-1} - g_{n+1}^{-1}g_n g_{n-1}^{-1} \right ) g_{n-1} \\
&= g_{n} \left (	g_n^{-1} - h_{\sigma_n}^{-1}\left (g_n^{-1}\right ) \right ) g_{n-1} \\
&= g_{n} \left (	g_n^{-1} - h_{\sigma_n}^{-1}\left (h_{\iota(\beta)}^{-1}(g_n^{-1})\right ) \right ) g_{n-1} \\
&= g_{n} \left (	g_n^{-1} - h_{\beta_+}^{-1}\left (g_n^{-1}\right ) \right ) g_{n-1} \\
&= g_{n} \left (	h_{\beta_+}\left (h_{\beta_+}^{-1}\left (g_n^{-1}\right )\right ) - h_{\beta_+}^{-1}\left (g_n^{-1}\right ) \right ) g_{n-1}.
\end{align*}
Hence, by composing with $\gamma_{\beta_+}\colon \Z[\F^n] \to  \Z[G_{\beta_+}]$, the quotient by all relations 
$h_{\beta_+}(*) = *$,
we get
$
\gamma_{\beta_+}\left (		g_{n-1}- g_n g_{n+1}^{-1}g_{n}	\right ) =
0.
$
The previous equality to $0$ remains true through any deeper epimorphism $Q_{\beta_+}=\psi_{\beta_+} \circ \gamma_{\beta_+}$.
\item Applying $\det^r_{\Gamma_{\psi_{\beta_+}}}$ to the final equality now yields (as expected):
\begin{align*}
&\max(1,t)^{n}  \cdot 
\det{}^r_{\Gamma_{\psi_{\beta_+}}}\left (
\overline{\mathcal{B}}^{(2)}_{t,Q_{\beta_+}}
(\beta_+) - \Id^{\oplus n}
\right )
\cdot \frac{1}{t}\\
&= 	\det{}^r_{\Gamma_{\psi_{\beta_+}}}\left (
\overline{\mathcal{B}}^{(2)}_{t,Q_{\beta_+}\circ\iota_{\F_n}}
(\beta) - \Id^{\oplus (n-1)}
\right ) \cdot 
 \frac{1}{t} \max(1,t)^{n+1}.
\end{align*}
\end{itemize}
\end{proof}

\begin{remark}
	One can compare the techniques used in the proof of Proposition \ref{prop:second} with  those used in \cite[Section 3.3]{Bi} and \cite[Section 3.4.1]{KT}. In particular, the construction of the matrix $\mathfrak{G}$ in the previous proof resembles the simpler ones $C$ of \cite[Section 3.3]{Bi}  and $C_n$ of \cite[Section 3.4.1]{KT}, which makes sense since their properties all come from the fundamental formula of Fox calculus.
\end{remark}

\subsection{Counter-examples to Markov invariance}\label{sec:counter:ex}

Theorem \ref{thm:burau:alexander:twisted} established that $F_\QQ$ is a Markov function when the Markov-admissible family of epimorphisms $\QQ$ descends to the link groups or lower. It is now natural to ask if those families $\QQ$ are the only ones for which $F_\QQ$ is a Markov function.
We do not have an answer to this question at the time of writing.

However, by looking at the details of the proofs of Propositions \ref{prop:first} and \ref{prop:second}, it appears that applying the epimorphism $\gamma_\beta$ (or a strictly deeper epimorphism) is necessary in order to obtain the cancellations in the matrices that yield Markov invariance for $F_\QQ$. As additional evidence for this hypothesis, we discovered two families $\QQ$ such that $F_\QQ$ is not a Markov function:
one family (the identities of the free group) lives strictly higher than the $\gamma_\beta$, and the other one (the abelianizations of the free groups) is not comparable to the $\gamma_\beta$ (Figure \ref{fig:epimorphisms} can help visualizing this). See the following Theorems \ref{thm:contrex:abel} and \ref{thm:contrex:id}.

These two counter-examples illustrate some of the difficulties in computing general Fuglede--Kadison determinants, and some transversal techniques one might need to use in order to do so (techniques such as computation of Mahler measures of polynomials, or combinatorics of closed paths on Cayley graphs). Notably, the proof of Theorem \ref{thm:contrex:id} uses   the new values for Fuglede--Kadison determinants over free groups computed in \cite[Theorem 1.2]{BA6} via combinatorial and analytical techniques.

Of course, appearances might still deceive, and it could happen that unexpected identities of Fuglede--Kadison determinants occur even with families of epimorphisms not encompassed by Theorem \ref{thm:burau:alexander:twisted}, especially since computing such determinants remains a daunting task today.

In the following theorem, we prove that when the family  $\QQ$ descends to the free abelian groups, $F_\QQ$ is not a Markov function and thus  cannot yield link invariants. The proof uses properties of the Fuglede--Kadison determinant and a value of the Mahler measure due to Boyd \cite{Bo}.

\begin{theo}\label{thm:contrex:abel}
For the family of abelianizations $\mathcal{Q} =\left \{ \varphi_{n(\beta)}\colon \F_{n(\beta)} \twoheadrightarrow \Z^{n(\beta)} \right \}$, the value at $t=1$ of the function $F_{\QQ}$ is not invariant under Markov moves.	In particular $F_{\QQ}$ is not a Markov function.
\end{theo}

\begin{proof}
Let
$t>0, \ n=2, \ \beta = \sigma_1^{-1} \in B_2$, and  $\beta_+ =  \sigma_1^{-1} \sigma_2\in B_3$.
Following Section \ref{subsec:braid}, we compute that $\beta_+$ acts on $\F_3$ by
$$\left \{ \begin{matrix}
g_1 \\ g_2
\end{matrix}\right . \overset{h_{\beta_+}}{\mapsto}
\left \{ \begin{matrix}
g_1^{-1} g_2 \\ g_3 g_2^{-1}g_1^{-1} g_2
\end{matrix},\right .
$$
where $g_1, g_2, g_3$ denote the generators as in Section \ref{subsec:braid}.
Thus Fox calculus gives:
$$\overline{\mathcal{B}}^{(2)}_{t,id}(\beta_+) - \Id^{\oplus 2} =
\begin{pmatrix}
-\frac{1}{t} R_{g_1^{-1}} - \Id & -R_{g_3 g_2^{-1}g_1^{-1}} \\
\frac{1}{t} R_{g_1^{-1}} & -\Id - t R_{g_3 g_2^{-1}} + R_{g_3 g_2^{-1}g_1^{-1}}
\end{pmatrix}.
$$

Hence, from Remark \ref{rem:atiyah} and Proposition \ref{prop:detFK:properties} (6), we have:
\begin{align*}
&\det{}_{\F_3}\left (\overline{\mathcal{B}}^{(2)}_{t,id}(\beta_+) - \Id^{\oplus 2}\right ) \\
& = \det{}_{\F_3}\left (
\left ( -\Id - t R_{g_3 g_2^{-1}} + R_{g_3 g_2^{-1}g_1^{-1}}\right )
\left ( -R_{g_1 g_2 g_3^{-1}}\right )
\left (-\frac{1}{t} R_{g_1^{-1}} - \Id\right )
-\frac{1}{t} R_{g_1^{-1}} 
\right )\\
&=\det{}_{\F_3}\left (
-\left (
{t} R_{g_1}+R_{g_1 g_2 g_3^{-1}}+ \frac{1}{t} R_{g_2 g_3^{-1}}
\right )
\right )\\
& =\frac{1}{t}\det{}_{\F_3}\left (
\Id + t R_{g_1} + t^2 R_{g_1 g_3 g_2^{-1}}
\right ).
\end{align*}

Let $z_1,z_2,z_3$ denote the canonical generators of $\Z^3$. Then it follows from the same arguments as in the previous computation that:
$$\det{}_{\F_3}\left (\overline{\mathcal{B}}^{(2)}_{t,\varphi_3}(\beta_+) - \Id^{\oplus 2}\right ) = \frac{1}{t}\det{}_{\Z^3}\left (
\Id + t R_{z_1} + t^2 R_{z_1 z_3}
\right ).$$
Let $H$ be the subgroup of $\Z^3$ isomorphic to $\Z^2$ and generated by $\lambda=z_1$ and $\mu=z_1 z_3$. Recall that $F_{\QQ}(\beta_+)$ is an equivalence class of functions of $t>0$ up to multiplication by a monomial, thus the value at $t=1$ is always the same regardless of the representant of the equivalence class. Let us denote this value $F_{\QQ}(\beta_+)(1)$. We then have:
\begin{align*}
F_{\QQ}(\beta_+)(1)&=\det{}_{\Z^3}\left (
\Id +  R_{z_1} + R_{z_1 z_3}
\right )\\
&=\det{}_{H}\left (
\Id + R_{\lambda} + R_{\mu}
\right ) \\
&= \mathcal{M}(1+X+Y) = 1.38135...  \neq 1.
\end{align*}
where $\mathcal{M}$ is the Mahler measure, the second equality follows from Proposition \ref{prop:detFK:properties} (3), the third one from Proposition \ref{prop:detFK:properties} (7) and the fourth one  from
Example \ref{ex:mahler:boyd}.

Now, since by Proposition \ref{prop:detFK:properties} (5) we have:
$$F_{\QQ}(\beta)(1)= \det{}_\Z\left (-R_{\varphi_2(g_1^{-1}g_2)}-\Id\right )=
1,$$ we conclude that $\beta \mapsto F_{\QQ}(\beta)(1)$ is not invariant under Markov moves.
\end{proof}

We cannot expect Markov invariance by remaining at the level of the free group either, as the following theorem shows:

\begin{theo}\label{thm:contrex:id}
	For the family of identities $\mathcal{Q} =\{ \id_{\F_{n(\beta)}} \}$, the function $F_{\QQ}$ is not invariant under  Markov moves.	
\end{theo}

\begin{proof}
Let us take $t=1, \ n=2, \ \beta = \sigma_1^{-1} \in B_2, \ \beta_+ = \sigma_1^{-1}  \sigma_2\in B_3$. Then, as in the proof of Theorem \ref{thm:contrex:abel}, we obtain
\begin{align*}
F_{\QQ}(\beta_+)(1)&=\det{}_{\F_3}\left (\mathcal{B}^{(2)}_{1,id}(\beta_+) - \Id^{\oplus 2}\right ) \\
&=\det{}_{\F_3}\left (
\Id +  R_{g_1} + R_{g_1 g_3 g_2^{-1}}
\right )\\
&=\det{}_{F}\left (
\Id +  R_{x} + R_{y}
\right ),
\end{align*}
where $F$ is the free group on two generators $x,y$ that embeds in $\F_3$ via $x \mapsto g_1, y \mapsto g_1 g_3 g_2^{-1}$, and the last equality follows from Proposition \ref{prop:detFK:properties} (3).
Now, since 
 $$F_{\QQ}(\beta)(1)= \det{}_{\F_2}\left (-R_{g_1^{-1}g_2}-\Id\right )=
1,$$ it remains to prove that 
$\det_{F}\left (
\Id +  R_{x} + R_{y}
\right ) \neq  1$. This follows from \cite[Theorem 1.2]{BA6}, which establishes that
$$\det{}_{F}\left (
\Id +  R_{x} + R_{y}
\right ) = \dfrac{2}{\sqrt{3}} \neq 1.$$
\end{proof}

\section*{Acknowledgements}
The author was supported by the FNRS in his "Research Fellow" position at UCLouvain, under Grant
no. 1B03320F. We thank Cristina Anghel-Palmer, Anthony Conway, Jacques Darné and Wolfgang Lück for helpful discussions, and the anonymous referee for useful suggestions.

\end{document}